\documentclass[11pt]{article}
\usepackage{amsmath,amssymb,amsthm,amscd,mathdots}

\newtheorem{theorem}{Theorem}[section]
\newtheorem{lemma}[theorem]{Lemma}
\newtheorem{proposition}[theorem]{Proposition}
\newtheorem{corollary}[theorem]{Corollary}

\theoremstyle{plain}

\theoremstyle{definition}

\numberwithin{equation}{section}

\renewcommand{\theenumi}{(\roman{enumi})}
\renewcommand{\labelenumi}{\textup{(\theenumi)}}

\title{Extension groups for the $C^*$-algebras associated with \\
$\lambda$-graph systems 
}
\author{Kengo Matsumoto \\
Department of Mathematics \\
Joetsu University of Education \\
Joetsu, 943-8512, Japan
}

\begin{document}

\maketitle

\date{}

\def\det{{{\operatorname{det}}}}

\begin{abstract} A $\lambda$-graph system is a labeled Bratteli diagram 
with certain additional structure, which presents a subshift.
The class of the $C^*$-algebras $\mathcal{O}_{\frak L}$ associated with 
the $\lambda$-graph systems is a generalized class of the class of Cuntz--Krieger algebras.  
In this paper, we will compute the strong extension groups
$\operatorname{Ext}_{\operatorname{s}}(\mathcal{O}_{\frak L})$ 
for the $C^*$-algebras associated with $\lambda$-graph systems ${\frak L}$
and study their relation with the weak extension group 
$\operatorname{Ext}_{\operatorname{w}}(\mathcal{O}_{\frak L})$.
\end{abstract}





\newcommand{\Ker}{\operatorname{Ker}}
\newcommand{\sgn}{\operatorname{sgn}}
\newcommand{\Ad}{\operatorname{Ad}}
\newcommand{\ad}{\operatorname{ad}}
\newcommand{\orb}{\operatorname{orb}}

\def\Re{{\operatorname{Re}}}
\def\det{{{\operatorname{det}}}}
\def\calK{{\mathcal{K}}}

\newcommand{\N}{\mathbb{N}}
\newcommand{\C}{\mathbb{C}}
\newcommand{\R}{\mathbb{R}}
\newcommand{\Rp}{{\mathbb{R}}^*_+}
\newcommand{\T}{\mathcal{T}}

\newcommand{\Z}{\mathbb{Z}}
\newcommand{\Zp}{{\mathbb{Z}}_+}

\def\K{{{\operatorname{K}}}}
\def\Ext{{{\operatorname{Ext}}}}
\def\Exts{{{\operatorname{Ext}_{\operatorname{s}}}}}
\def\Extw{{{\operatorname{Ext}_{\operatorname{w}}}}}
\def\Ext{{{\operatorname{Ext}}}}
\def\KK{{{\operatorname{KK}}}}

\def\OA{{{\mathcal{O}}_A}}
\def\ON{{{\mathcal{O}}_N}}

\def\OSG{{{\mathcal{O}}_{S_G}}}

\def\A{{\mathcal{A}}}

\def\B{{\mathcal{B}}}
\def\E{{\mathcal{E}}}
\def\Q{{\mathcal{Q}}}
\def\calQ{{\mathcal{Q}}}

\def\calC{{\mathcal{C}}}

\def\Ext{{{\operatorname{Ext}}}}
\def\Max{{{\operatorname{Max}}}}
\def\min{{{\operatorname{min}}}}
\def\max{{{\operatorname{max}}}}
\def\KMS{{{\operatorname{KMS}}}}
\def\Proj{{{\operatorname{Proj}}}}
\def\Out{{{\operatorname{Out}}}}
\def\Aut{{{\operatorname{Aut}}}}
\def\Ad{{{\operatorname{Ad}}}}
\def\Inn{{{\operatorname{Inn}}}}
\def\Int{{{\operatorname{Int}}}}
\def\det{{{\operatorname{det}}}}
\def\exp{{{\operatorname{exp}}}}
\def\nep{{{\operatorname{nep}}}}
\def\sgn{{{\operatorname{sign}}}}
\def\cobdy{{{\operatorname{cobdy}}}}

\def\Ker{{{\operatorname{Ker}}}}
\def\Coker{{{\operatorname{Coker}}}}
\def\Im{{\operatorname{Im}}}
\def\Hom{{\operatorname{Hom}}}

\newcommand{\Ew}{{\operatorname{Ext_{\operatorname{w}}}}}
\def\Es{{{\operatorname{Ext_{\operatorname{s}}}}}}

\def\Extwz{{{\operatorname{Ext_{\operatorname{w}}^0}}}}
\def\Extsz{{{\operatorname{Ext_{\operatorname{s}}^0}}}}
\def\Extwo{{{\operatorname{Ext_{\operatorname{w}}^1}}}}
\def\Extso{{{\operatorname{Ext_{\operatorname{s}}^1}}}}
\def\Extwi{{{\operatorname{Ext_{\operatorname{w}}^i}}}}
\def\Extsi{{{\operatorname{Ext_{\operatorname{s}}^i}}}}

\def\B{{\mathcal{B}}}
\def\D{{\mathcal{D}}}
\def\G{{\mathcal{G}}}

\def\LLL{{ {\frak L}^{\lambda(\Lambda)} }}
\def\LLTL{{ {\frak L}^{\lambda(\widetilde{\Lambda})} }}

\def\LSG{{ {\frak L}_{S_G} }}

\def\OA{{ {\mathcal{O}}_A }}
\def\OB{{ {\mathcal{O}}_B }}
\def\DA{{ {\mathcal{D}}_A }}
\def\DB{{ {\mathcal{D}}_B }}

\def\OL{{ {\mathcal{O}}_{\frak L}}}
\def\DL{{ {\mathcal{D}}_{\frak L}}}
\def\AL{{ {\mathcal{A}}_{\frak L}}}
\def\OLmin{{ {\mathcal{O}}_{{\Lambda}^{\operatorname{min}}} }}

\def\DLmin{{ {\mathcal{D}}_{{\frak L}_\Lambda^{\operatorname{min}}} }}

\def\DLam{{ {\mathcal{D}}_{\Lambda}}}

\def\L{{\frak L}}

\def\A{ {\mathcal{A}} }
\def\D{ {\mathcal{D}} }
\def\F{ {\mathcal{F}} }
\def\S{{\mathcal{S}}}

\def\Ext{{{\operatorname{Ext}}}}
\def\Im{{{\operatorname{Im}}}}
\def\Aut{{{\operatorname{Aut}}}}
\def\Ad{{{\operatorname{Ad}}}}

\def\Int{{{\operatorname{Int}}}}

\def\Hom{{{\operatorname{Hom}}}}
\def\Ker{{{\operatorname{Ker}}}}
\def\dim{{{\operatorname{dim}}}}
\def\min{{{\operatorname{min}}}}
\def\Max{{{\operatorname{Max}}}}

\def\id{{{\operatorname{id}}}}
\def\ind{{{\operatorname{ind}}}}
\def\Ind{{{\operatorname{Ind}}}}

\def\Indw{{{\operatorname{Ind_{\operatorname{w}}}}}}
\def\Inds{{{\operatorname{Ind_{\operatorname{s}}}}}}

\def\Indwz{{{\operatorname{Ind_{\operatorname{w}}^0}}}}
\def\Indsz{{{\operatorname{Ind_{\operatorname{s}}^0}}}}
\def\Indwo{{{\operatorname{Ind_{\operatorname{w}}^1}}}}
\def\Indso{{{\operatorname{Ind_{\operatorname{s}}^1}}}}

\def\LDNmin{{ {\frak L}_{D_N^{\operatorname{min}}} }}
\def\LDtmin{{ {\frak L}_{D_2^{\operatorname{min}}} }}

\def\ODNmin{{ {\mathcal{O}}_{D_N^{\operatorname{min}}} }}
\def\ODtmin{{ {\mathcal{O}}_{D_2^{\operatorname{min}}} }}
\section{Introduction}
Throughout the paper, 
$B(H)$ denotes the $C^*$-algebra of bounded linear operators on
a separable infinite-dimensional 
Hilbert space $H$.
Let us denote by $K(H)$ the $C^*$-subalgebra of $B(H)$
of compact operators on $H$, which is a closed two-sided ideal of 
$B(H)$. 
The quotient $C^*$-algebra $B(H)/K(H)$ is called the Calkin algebra, denoted by $Q(H)$.
The quotient map $B(H)\longrightarrow Q(H)$ is denoted by $\pi.$  
Let $\A$ be a separable unital nuclear $C^*$-algebra.
The extension group $\Ext_*(\A)$ is defined by equivalence classes of 
short exact sequences 
\begin{equation}\label{eq:exact}
0 
\longrightarrow K(H) 
\longrightarrow \E 
\longrightarrow \A 
\longrightarrow 0 \qquad (\text{exact})
\end{equation} 
of $C^*$-algebras for which $K(H)$ is an essential ideal of $\E$.
There are two kinds of extension groups 
$\Exts(\A)$ and $\Extw(\A)$,
called the strong extension group and the weak extension group, respectively.
They are defined by
two different equivalence relations of the short exact sequences
\eqref{eq:exact}, respectively.
The groups have been playing important role as one of K-theoretic invatriant 
 in studying structure theory of $C^*$-algebras, classification of 
 essentially normal operators,  non commutative geometry, and so on. 
In \cite{CK}, Cuntz--Krieger have computed the weak extension group
$\Extw(\OA)$ of Cuntz--Krieger algebra $\OA$ as
$\Extw(\OA) = \Z^N/(I -A)\Z^N$ for 
an $N \times N$ irreducible matrix with entries in $\{0,1\},$
so that they found a lot of examples of 
unital simple purely infinite $C^*$-algebras which are mutually non-isomorphic.
On the other hand, Paschke--Salinas \cite{PS} 
and Pimsner--Popa \cite{PP}
independently computed 
the groups $\Extw(\mathcal{O}_N)$ and 
$\Exts(\mathcal{O}_N)$ for Cuntz algebras as 
$\Extw(\mathcal{O}_N) =\Z/(1 - N)\Z$ and $\Exts(\mathcal{O}_N) =\Z$.
It is a remarkable fact that the group
$\Z^N/(I -A)\Z^N$ appears as 
$\Extw(\OA)$ from the view point of 
classification theory of symbolic dynamical systems,
because the group $\Z^N/(I -A)\Z^N$ is an 'almost' complete invariant 
of flow equivalence of the associated two-sided topological Markov shift
$(\Lambda_A, \sigma_A)$ (\cite{Franks}, \cite{BF}).

In \cite{MaPre2021exts} (cf. \cite{MaJMAA2024}), 
the author recently computed the strong extension group
$\Exts(\OA)$ of Cuntz--Krieger algebras as
$\Exts(\OA) = \Z/(I-\widehat{A})\Z^N$, where
$\widehat{A} = A + R_1 - AR_1$ and $R_1 $ is the $N \times N$ matrix whose first row
is $[1,1,\dots, 1]$ and the other rows are zero vectors. 
In \cite{MaPre2021exts}, 
the author also clarified exact relationship between the two groups
$\Extw(\OA)$ and $\Exts(\OA)$ 
by presented the following cyclic six term exact sequence of abelian groups:
\begin{equation}\label{eq:6termA}
\begin{CD}
 \Ker(I - \widehat{A}: \Z^N \longrightarrow \Z^N) /\iota(\Z) 
 @>{}>> \Ker(I - A: \Z^N \longrightarrow \Z^N) @>>> \Z \\
@AAA  @.   @VV{}V  \\
0 @<{}<<  \Z^N/(I - A)\Z^N  @<{}<<  \Z^N/(I - \widehat{A})\Z^N  
\end{CD}
\end{equation}
by translating an associated  $\K$-homology long exact sequence (cf. \cite{HR}),
where $\iota: \Z \rightarrow \Ker(I - \widehat{A}: \Z^N \longrightarrow \Z^N)$
is given by $\iota(m) = [m,0.\dots,0], \, m \in \Z$.
The importance of the two groups $\Exts(\OA)$ and $\Extw(\OA)$ 
 in the classification theory of $C^*$-algebras was shown in a recent joint paper
 \cite{MaSogabe2024a}. 
 It says that the two groups $\Exts(\OA)$ and $\Extw(\OA)$ are
 a complete set of invariants  of the
 isomorphism class of the Cuntz--Krieger algebra $\OA$.
 That is, the position $[1]_0$ in $\K_0(\OA)$ of the unit of the $\OA$
 is determined by only the group structure of
 the two extension groups $\Exts(\OA)$ and $\Extw(\OA)$.
 The result will be generalized to a wider class of Kirchberg algebras
 in \cite{MaSogabe2024b}.
Hence it seems to be interesting and important to  compute 
the two extension groups $\Exts(\, \cdot \, ) $ and $\Extw(\, \cdot \,)$ 
for more general Kirchberg algebras.

In this paper we will generalize the above computations for  
$\Exts(\OA)$ to  more general $C^*$-algebras 
related to symbolic dynamical systems.
Cuntz--Krieger algebra $\OA$ 
is considered as the $C^*$-algebra associated to topological Markov shift
defined by the matrix $A$.
The $C^*$-algebras 
which we will consider in the present paper
are the ones associated to general subshifts defined in \cite{MaDocMath2002} 
(cf. \cite{MaDocMath2023}).
They are defined by 
$\lambda$-graph systems ${\frak L}$ 
which are labeled Bratteli diagrams with some additional structure,
and are regarded as 
generalizations of finite directed graphs. 
Let $\Sigma$ be a finite set whose cardinality 
$|\Sigma| \ge 2$.
A $\lambda$-graph sytems 
$ \frak L = (V,E,\lambda,\iota)$ over alphabet $\Sigma$
consists of the vertex set 
$V = \cup_{l=0}^\infty V_l$,
edge set
 $E = \cup_{l=0}^\infty E_{l,l+1}$,
 labeling map
$\lambda: E\longrightarrow \Sigma$
and a sequence $\iota =\{ \iota_{l,l+1}\}$ of surjective maps 
$\iota_{l,l+1}:V_{l+1}\longrightarrow V_l$ 
for each $l \in \Zp =\{ 0,1,2, \dots \}$.
Each edge  $e \in E_{l,l+1}$ 
has 
its source vertex $s(e) \in V_l$
and its terminal vertex $t(e) \in V_{l+1}$,
 with its label
$\lambda(e ) \in \Sigma$,
so $\lambda$ is a map from $E$ to $\Sigma$.
The first three $(V,E,\lambda)$ expresses a labeled Bratteli diagram.
Moreover the extra structure 
$\iota_{l,l+1}: V_{l+1} \longrightarrow V_l, \, l \in \Zp$
 of surjective maps satisfies 
 a property called  the local property of $\lambda$-graph system
such that 
for every two vertices
$u \in V_l$ and  $v \in V_{l+2}$,
there exists a bijective correspondence 
preserving their labels between 
the two sets  of edges 
$$
 \{  
e \in E_{l,l+1} \mid s(e) = u, t(e) = \iota(v) \},\quad
\{  e \in E_{l+1,l+2} \mid  \iota(s(e)) = u, t(e) = v \}.
$$
Put 
$m(l) = |V_l|$
the cardinality $| V_l |$ of the finite set $V_l$, so that 
$m(l) \le m(l+1)$.
A $\lambda$-graph system $\L$ is said to be {\it left-resolving}\/ 
if
$e,f \in E_{l,l+1}$ satisfy $\lambda(e) = \lambda(f),
t(e) = t(f)$, 
then  $e=f$.
We henceforth assume that a $\lambda$-graph system is left-resolving.

Any $\lambda$-graph system defines a subshift by 
gathering label sequences appearing on the concatenated labeled edges of the $\lambda$-graph system.
Conversely any subshift can be presented by a $\lambda$-graph system (\cite{MaDocMath1999}).
Hence a $\lambda$-graph system is regarded as a graph presentation of a subshift.

We fix a left-resolving $\lambda$-graph system 
${\frak L} = (V, E, \lambda,\iota)$ over $\Sigma$.
Let us denote by the vertex set $V_l =\{v_1^l,\dots, v_{m(l)}^l\}$.
Let $(A_{l,l+1},I_{l,l+1})_{l\in \Zp}$ 
be the structure matrices of a given $\lambda$-graph system
$\frak L$ defined by for $\alpha\in \Sigma$ 
\begin{align*}
A_{l,l+1}(i, \alpha,j) = &
{\begin{cases}
1 & \text{ if } \exists e \in E_{l,l+1}; s(e) = v_i^l, t(e) = v_j^{l+1}, \lambda(e) =\alpha, \\
0 & \text{ otherwise},  
\end{cases}} \\
 I_{l,l+1}(i,j)  = & 
 {\begin{cases}
1 & \text{ if  } \iota(v_j^{l+1}) = v_i^l, \\
0 & \text{ otherwise},  
\end{cases}}
\end{align*}
so  $A_{l,l+1}, I_{l,l+1}$
are $m(l)\times m(l+1)$ matrices.
\begin{proposition}[{\cite{MaDocMath2002}},\cite{MaJMAA2021}] \label{prop:relationL}
The $C^*$-algebra $\OL$ is realized as a universal $C^*$-algebra 
generated by partial isometries $S_\alpha$ indexed by $\alpha \in \Sigma$
and mutually commuting projections
$E_i^l$ indexed by $v_i^l \in V_l$ 
subject to the following operator relations called $({\frak L})$:
\begin{gather*}
1  = \sum_{\beta\in \Sigma} S_\beta S_\beta^* = \sum_{i=1}^{m(l)} E_i^l, \qquad
E_i^l  = \sum_{j=1}^{m(l+1)} I_{l,l+1}(i,j)E_j^{l+1}, \\ 
S_\alpha^* E_i^l S_\alpha 
= \sum_{j=1}^{m(l+1)} A_{l,l+1}(i,\alpha,j)E_j^{m(l+1)},\quad \alpha \in \Sigma, \,\, 
i=1,\dots,m(l), l\in \Zp.
\end{gather*}
\end{proposition}
The $C^*$-algebra $\OL$ 
was primarily constructed in \cite{MaDocMath2002} as the $C^*$-algebra $C^*(G_{\frak L})$
of an \'etale amenable groupoid $G_{\frak L}$ associated to the $\lambda$-graph system $\frak L$. 
The class of the $C^*$-algebras $\OL$ generalize the class of Cuntz--Krieger algebras $\OA$.
If $\frak L$ satisfies condition $(I)$, 
it is a unique $C^*$-algebra subject to the above operator relations $({\frak L})$.
It becomes a simple $C^*$-algebra if ${\frak L}$ satisfies condition $(I)$
and is irreducible 
({\cite{MaDocMath2002}},\cite{MaJMAA2021}).
As in \cite{MaJMAA2021}, 
lots  of the $C^*$-algebras $\OL$ are unital Kirchberg algebras.

Let $(A_{l,l+1},I_{l,l+1})_{l\in \Zp}$ 
be the structure matrices of a $\lambda$-graph system
${\frak L} = (V, E, \lambda,\iota)$.
Let
$$
{\Z}_I: =
\{ ([n_i^l]_{i=1}^{m(l)})_{l\in \Zp} \in \prod_{l\in \Zp} \Z^{m(l)}
\mid I_{l,l+1}[n_j^l]_{i=1}^{m(l+1)}= [n_i^l]_{i=1}^{m(l)}, l \in \Zp \}
$$
the projective limit
$\varprojlim \{I_{l,l+1}: 
{\Z}^{m(l+1)} \rightarrow {\Z}^{m(l)} \}
$ of the projective system 
$I_{l.l+1}: \Z^{m(l+1)}\rightarrow \Z^{m(l)}, l\in \Zp$
of abelian groups.
The subgroup 
${\Z}_{I,0}$ of ${\Z}_I$ is defined by 
\begin{equation*}
\Z_{I,0} = \{([n_i^l]_{i=1}^{m(l)})_{l\in \Zp} \in \Z_I 
\mid \sum_{i=1}^{m(l)} n_i^l =0, l\in \Zp \}.
\end{equation*}
The family $I_{l,l+1}-A_{l,l+1}, l \in \Zp$ 
of $m(l)\times m(l+1)$ matrices $I_{l,l+1}-A_{l,l+1}$
naturally give rise to an endomorphism on $\Z_I$ denoted by $I-A_\L$.
It satisfies 
$(I - A_\L)(\Z_{I,0}) \subset \Z_I$.
Following Higson--Roe \cite{HR}, 
the reduced $\K$-homology groups
$\widetilde{K}^i(\A), i=0,1$
and the unreduced $\K$-homology groups
${K}^i(\A), i=0,1$ 
for a separable $C^*$-algebra $\A$ are defined 
by using the dual $C^*$-algebra $\frak{D}(\A)$ (cf. \cite{Paschke1981})
for $\A$ such that 
$$
\widetilde{\K}^p(\A) = \K_{1-p}(\frak{D}(\A)) \quad\text{ and } \quad
\K^p(\A) = \K_{1-p}(\frak{D}(\widetilde{\A}))
$$
for $p =0,1$, where 
$\widetilde{\A}$ is the unitization of $\A$. 
They satisfy
$$
\widetilde{\K}^1(\A) = \Exts(\A) \quad\text{ and } \quad
\K^1(\A) = \Extw(\A).
$$
We write 
$\widetilde{\K}^i(\A) =: \Extsi(\A) 
$
and 
${\K}^i(\A)=:\Extwi(\A)
$
for $i=0,1$.
There exists a cyclic six-term exact sequence:
\begin{equation}\label{eq:Khom6}
\begin{CD}
 \Extsz(\A) @>>> \Extwz(\A) @>>> \Z \\
@AAA @. @VVV \\
0 @<<< \Extwo(\A) @<<< \Extso(\A)  
 \end{CD}
\end{equation}
for a separable unital nuclear $C^*$-algebra $\A$
 (\cite[5.2.10 Proposition]{HR}).

In the first half of the  paper, we will prove the following theorem.
\begin{theorem}[{Theorem \ref{thm:main11}}]\label{thm:main1}
Let $\L$ be a left-resolving $\lambda$-graph system over $\Sigma$.
There exist isomorphisms
\begin{align*}
\Indwo: \Extwo(\OL) & \longrightarrow \Z_I/(I - A_\L)\Z_I,\\
\Indso: \Extso(\OL) & \longrightarrow \Z_I/(I - A_\L)\Z_{I,0},\\
\Indwz: \Extwz(\OL) & \longrightarrow \Ker(I - A_\L:\Z_I\longrightarrow \Z_I),\\
\Indsz: \Extsz(\OL) & \longrightarrow \Ker(I - A_\L: \Z_{I,0}\longrightarrow \Z_I)
\end{align*} 
of abelian groups such that the $\K$-homology long exact sequence
\eqref{eq:Khom6} for $\A = \OL$ is given by
the cyclic six-term exact sequence
\begin{equation}\label{eq:6termL}
\begin{CD}
 \Ker(I - A_\L: \Z_{I,0}\longrightarrow \Z_I) 
 @>{}>> \Ker(I - A_\L: \Z_I\longrightarrow \Z_I) @>>> \Z \\
@AAA  @.   @VV{}V  \\
0 @<<< \Z_I/(I - A_\L)\Z_I  @<{}<< \Z_I/(I - A_\L)\Z_{I,0}.
\end{CD}
\end{equation}
\end{theorem}
In the second half of the paper, 
we will compute the extension groups $\Extsi(\OL)$
and the cyclic six-term exact sequence \eqref{eq:6termL}
for two kinds of examples of the $\lambda$-graph systems associated to subshifts.
There are lots of examples of subshifts which are not topological Markov shifts.
The first example of subshifts is the Markov coded systems $S_G$  studied in 
\cite{MaActaSci2014}.
The Markov coded system $S_G$ is defined by a finite directed graph
$G = (V, E)$ admitting multiple directed edges from a vertex to another vertex. 
Let $\{v_1,\dots, v_N\}$ be the vertex set $V$.
Let us denote by $A=[A(i,j)]_{i,j=1}^N (= A_G)$ the $N \times N$ transition matrix
of the directed graph such that $A(i,j)$ denotes the number of the directed edges from 
the vertex $v_i$ to the vertex $v_j$.  
In \cite{MaActaSci2014}, it was proved that the $C^*$-algebra 
$\OSG$ for the canonically constructed $\lambda$-graph system $\LSG$ for the subshift $S_G$ 
is a unital purely infinite simple nuclear $C^*$-algebra,
and its $\K$-groups and the weak extension groups were computed.
Since the torsion free part of $\K_0(\OSG)$ is not isomorphic to $\K_1(\OSG)$,
the $C^*$-algebras $\OSG$ are never stably isomorphic to any of Cuntz--Krieger algebras.

The second example of subshifts are the class of Dyck shifts $D_N, 2\le N \in \N$.
They are interesting family of subshifts 
coming from automata theory and formal language theory
which are located in the subshifts far from topological Markov shifts.
They have minimal presentations of $\lambda$-graph systems which yield 
unital simple purely infinite $C^*$-algebras 
having infinite generators of its $K$-theory groups,
so that they do not belong to the class of Cuntz--Krieger algebras.
We will compute the strong extension groups
for the two kinds of examples $\OSG$ and $\ODNmin$ of  $C^*$-algebras
in the following way.
\begin{theorem}[{Proposition \ref{prop:MarkovcodeExts}, Corollary \ref{cor:DyckNExts}}] 
\hspace{6cm}
\begin{enumerate}
\renewcommand{\theenumi}{\roman{enumi}}
\renewcommand{\labelenumi}{\textup{(\theenumi)}}
\item
Assume that the transition matrix $A$ of a finite directed graph $G$ is aperiodic.
Let $\OSG$ be the simple purely infinite $C^*$-algebra associated with the canonical 
$\lambda$-graph system for the Markov coded system $S_G$.
Then we have
\begin{align*}
\Extwz(\OSG) \cong & (\Ker(A) \text{ in } \Z^N) \oplus \Z^N, \qquad
\Extwo(\OSG)\cong \Z^N/A\Z^N, \\
\Extsz(\OSG)  \cong &  (\Ker(A) \text{ in } \Z^N) \oplus \Z^{N-1}, \quad 
\Extso(\OSG) \cong \Z^N/A\Z^N.
\end{align*}
\item
Let $\ODNmin$ be the simple purely infinite $C^*$-algebra
associated with the minimal presentation $ {\frak L}_{D_N^\min}$ of the Dyck shift $D_N$.
Then we have
\begin{align*}
\Extwz(\ODNmin) \cong & \Hom(C(\calC, \Z),\Z),\qquad
\Extwo(\ODNmin)\cong \Z/N\Z, \\
\Extsz(\ODNmin) \cong & \Hom(C(\calC, \Z),\Z), \qquad  
\Extso(\ODNmin)\cong \Z.
\end{align*} 
where $C(\calC, \Z)$ 
denotes the abelian group of integer valued continuous functions on a Cantor discontinuum $\calC$.
\end{enumerate}
\end{theorem}
We note that the computations
$
\Extwz(\OSG) \cong (\Ker(A) \text{ in } \Z^N) \oplus \Z^N,
\Extwo\OSG)\cong \Z^N/A\Z^N
$
and
$
\Extwz(\ODNmin) \cong \Hom(C(\calC, \Z),\Z),
\Extwo(\ODNmin)\cong \Z/N\Z
$
are already obtained in \cite{MaActaSci2014} and \cite{KMDocMath2003},
respectively.
In this paper, 
we will compute the strong extension groups
$\Extsz(\OSG), \Extso(\OSG)$
and
$\Extsz(\ODNmin)$ and $\Extso(\ODNmin).$

\section{$\Ext_*(\OL)$ and Fredholm indices}
In what follows, $H$ denotes a separable infinite dimensional Hilbert space.
Let $\A$ be a separable unital nuclear $C^*$-algebra.
 An extension means a unital $*$-monomorphism 
 $\sigma:\A\longrightarrow Q(H)$ from $\A$ to the Calkin algebra $Q(H)$.
 Two extensions
$\sigma_i : \A \longrightarrow Q(H), i=1,2$
are said to be 
weakly equivalent 
if there exists a unitary $u\in Q(H)$ 
such that 
$\sigma_2(a) = u \sigma_1(a) u^*, a \in \A$. 
If in particular the above unitary $u$ in $Q(H)$ is taken as $u=\pi(U)$ 
for some unitary $U \in B(H)$,
  then the extensions
$\tau_i : \A \longrightarrow Q(H), i=1,2$
are said to be  
strongly equivalent.
Let us denote by $\Ew(\A)$ and $\Es(\A)$ 
the weak equivalence classes 
and the strong equivalence classes of extensions of $\A$, respectively.
The class of a unital $*$-monomorphism
$\tau:\A\longrightarrow Q(H)$ in $\Es(\A)$ is denoted by 
$[\tau]_s$, and similarly 
$[\tau]_w$ in $\Ew(\A).$  
Fix an identification between $H \oplus H$ and $H$.
Through  an embedding
$Q(H) \oplus Q(H) \hookrightarrow Q(H)$
defined by the identification, 
the sum of extensions 
$\tau_1 \oplus \tau_2$ 
are defined by the direct sum $\tau_1\oplus \tau_2$.
It is well-known that both $\Es(\A)$ and $\Ew(\A)$ 
become abelian  semigroups,
and also they are abelian groups for nuclear $C^*$-algebra $\A$ 
(cf. \cite{Arveson}, \cite{BDF}, \cite{CE}, \cite{Douglas}, etc.).
They are called the strong extension group for $\A$ 
and the weak extension group for $\A$, respectively.
Let us denote by $q_A: \Es(\A) \longrightarrow \Ew(\A)$
the  natural quotient map.
As in \cite{HR} and \cite{PP}, 
there exists a homomorphism
$\iota_\A: \Z \longrightarrow \Es(\A)$
such that the sequence
\begin{equation}\label{eq:ZSW}
\Z
 \overset{\iota_\A}{\longrightarrow} \Es(\A)
 \overset{q_\A}{\longrightarrow} \Ew(\A)
\end{equation}
is exact at the middle, so
the quotient group $\Es(\A)/\iota_\A(\Z)$ 
is isomorphic to $\Ew(\A).$

For projections $e \in Q(H) = B(H)/K(H)$
and $E \in B(H)$ with $\pi(E) =e$,
let $t \in Q(H)$ be an element such that 
$e t e$ is invertible in $eQ(H)e.$
Let $T \in B(H)$ be a lift of $t$, which satisfies
$\pi(T) = t.$  
The Fredholm index of $ETE$ in $EH$ is
denoted by $\ind_et$.
The integer $\ind_et$ does not depend on the choice of
$E$ and $T$.

Let $\L =(V, E,\lambda,\iota)$ be
 a left-resolving $\lambda$-graph system over $\Sigma.$
Let $S_\alpha, \alpha \in \Sigma$ and $E_i^l, v_i^l \in V_l$ 
be the generating partial isometries and mutually commuting projections satisfying 
the relations $({\frak L})$ in Proposition \ref{prop:relationL}.
Let us denote by $\A_l$ the commutative $C^*$-subalgebra 
of $\OL$ generated by the projections
$E_i^l, i=1,\dots,m(l)$.
The commutative $C^*$-subalgebra 
$\AL$ is defined by the one generated by $\A_l, l\in \Zp$.
Since $\A_l \subset \A_{l+1}, l \in \Zp$,
The algebra $\AL$ is a commutative AF-algebra.
Hence the projections 
 $E_1^l, \dots, E_{m(l)}^l$ are  the  minimal projections 
in $\A_l$ satisfying 
$E_i^l = \sum_{j=1}^{m(l+1)}I_{l,l+1}(i,j) E_j^{l+1}$ 
for $i=1,\dots,m(l)$.

\begin{lemma}\label{lem:lem5.1}
For an extension $\sigma:\OL\longrightarrow Q(H)$,
there exists a trivial extension
$\tau:\OL\longrightarrow B(H)$ such that 
$\sigma|_{\AL} = \pi\circ \tau|_{\AL}$.
\end{lemma}
\begin{proof}
Let us denote by $\hat{\sigma}$
the restriction $\sigma|_{\AL}$ of $\sigma$ to the subalgebra $\AL$. 
As $\AL$ is a commutative AF-algebra,
the extension $\hat{\sigma}$ is trivial by \cite[1.15 Theorem]{BDF}.
Take a unital $*$-monomorphism  
$\rho:\OL\longrightarrow B(H)$
and put
$\tilde{\rho} = \pi \circ \rho:\OL\longrightarrow Q(H)$
and
$$
\hat{\rho} =\pi\circ\rho|_{\AL}:\AL\longrightarrow Q(H).
$$
As the extensions $\hat{\sigma}, \hat{\rho}$ are both trivial,
by Voiculescu's theorem \cite{Voiculescu}, 
there exists a unitary $U \in B(H)$
such that 
$$
\hat{\sigma}(x) = \pi(U) \hat{\rho}(x) \pi(U)^*, \qquad x \in \AL
$$
so that
$
{\sigma}(x) = \pi(U) \pi(\rho(x)) \pi(U)^*, x \in \AL.
$
 By putting $\tau = \Ad(U)\circ \rho:\OL\longrightarrow B(H)$,
 we have
 $\sigma =\pi\circ \tau$ on $\AL.$
\end{proof}
A Fredholm module over a $C^*$-algebra $\A$ 
means a pair $(u,\rho)$ of a unitary $u \in Q(H)$
and a $*$-homomorphism $\rho:\A\longrightarrow B(H)$ such that
$\pi(\rho(a)) u =u \pi(\rho(a))$ for all $a \in \A$.
It is well-known that 
the K-homology group $\K^0(\A)$ for $\A$ 
is realized as 
the homotopy equivalence classes of Fredholm modules over $\A$
(cf. \cite{BDF},  \cite{Douglas}, \cite{Higson}, \cite{HR}, etc.).
The addition
$[(u_1, \rho_1)] + [(u_2, \rho_2)] $
is defined by $[(u_1\oplus u_2,\rho_1\oplus  \rho_2)], $
in particular, we have
$
[(u_1, \rho)] + [(u_2, \rho)] = [(u_1 u_2, \rho)], 
$
 and hence
$
[(u^*, \rho)] + [(u, \rho)] = [(u^* u, \rho)]= [(1, \rho)], 
$
 so that 
 $-[(u,\rho)] = [(u^*,\rho)].$
 Take a faithful representation
 $\tau_0: \A\longrightarrow B(H)$ and a unitary
 $V_0 \in B(H)$ such that 
 $\tau_0(a)V_0 = V_0 \tau_0(a)$ for $a \in \A$.
 Since $(\pi(V_0),\tau_0)$ is a neutral element of $\K^0(\A)$,
 $(u\oplus \pi(V_0), \rho\oplus\tau_0)$ is equivalent to $(u,\rho)$
 for any Fredholm module $(u,\rho)$ over $\A$.
 Hence we can take a representative
of $[(u,\rho)]$ as $\rho$ being faithful.
As $\AL ={\displaystyle \lim_{l\to\infty}}\A_l$
is an AF-algebra, the formula
$$
\K^0(\AL) = \Hom(\K_0(\AL),\Z)= \varprojlim \K^0(\A_l) 
$$
holds (cf. \cite{Blackadar}, \cite{HR}).
For $[(u,\rho)] \in \K^0(\A_l),$
the identity
$E_i^l = \sum_{j=1}^{m(l+1)} I_{l,l+1}(i,j) E_j^{l+1}$
implies
\begin{equation*}
\ind_{\pi\circ\rho(E_i^l)}u 
=\sum_{j=1}^{m(l+1)} I_{l,l+1}(i,j) \ind_{\pi\circ\rho(E_j^{l+1})}u,
\end{equation*}
so that the 
 correspondence for each $l \in \Zp$
\begin{equation}\label{eq:indl}
\Ind_l: [(u,\rho)] \in \K^0(\A_l) \longrightarrow
(\ind_{\pi\circ\rho(E_i^l)}u)_{i=1}^{m(l)} \in \Z^{m(l)}
\end{equation}
yields the isomorphism
\begin{equation}\label{eq:KinjlimZ}
\K^0(\AL) 
= \varprojlim\{ I_{l,l+1}:\Z^{m(l+1)} \longrightarrow \Z^{m(l)}\}
\end{equation}
of abelian groups.
The latter group, denoted by $\Z_I$,
is the projective limit of the system
$I_{l,l+1}:\Z^{m(l+1)} \longrightarrow \Z^{m(l)}, l \in \Zp$. 
%
We denote by 
$
\Ind: \K^0(\AL) \longrightarrow \Z_I
$ the isomorphism
induced by \eqref{eq:indl}.
We note the following lemma.
\begin{lemma} \label{lem:5.2p}
Let $\rho_i: \AL\longrightarrow B(H), i=1,2$
be  representations and $u \in Q(H)$ a unitary
satisfying
$u \pi(\rho_i(a)) = \pi(\rho_i(a)) u, a \in \AL, i=1,2$.
Suppose that $\pi \circ \rho_2 = \pi\circ \rho_1$.
Then 
$[(u,\rho_1)] = [(u,\rho_2)]$ as elements of $\K^0(\AL)$. 
\end{lemma}
\begin{proof}
Since $\pi \circ \rho_2(E_i^l)  = \pi\circ \rho_1(E_i^l), i=1,\dots,m(l),$
we have
$\ind_{\pi\circ\rho_2(E_i^l)}u = \ind_{\pi\circ\rho_1(E_i^l)}u$ for $i=1,\dots, m(l)$, so that 
$[(u,\rho_1)] = [(u,\rho_2)]$.
\end{proof}
For an extension $\sigma:\OL\longrightarrow Q(H)$, 
take a trivial extension
$\tau:\OL\longrightarrow B(H)$ such that 
$\sigma|_{\AL} = \pi\circ \tau|_{\AL}$.
\def\ttau{\tilde{\tau}}
We write 
$\ttau = \pi\circ \tau: \OL\longrightarrow Q(H)$
and define
$U_{\sigma,\ttau} \in Q(H)$ by setting
\begin{equation*}
U_{\sigma,\ttau} := \sum_{\alpha \in \Sigma} \sigma(S_\alpha) \ttau(S_\alpha^*).
\end{equation*}
\begin{lemma}\label{lem:5.3}
$U_{\sigma,\ttau}$ is a unitary in $Q(H)$ 
such that 
$U_{\sigma,\ttau}\ttau(a) = \ttau(a) U_{\sigma,\ttau}$ for all $a \in \AL$.
Hence the pair
$(U_{\sigma,\ttau}, \tau|_{\AL})$ gives rise to an element of $\K^0(\AL)$.
\end{lemma}
\begin{proof}
We have
\begin{equation*}
U_{\sigma,\ttau} U_{\sigma,\ttau}^*
= \sum_{\alpha, \beta \in \Sigma} \sigma(S_\alpha) \ttau(S_\alpha^*)\ttau(S_\beta)\sigma(S_\beta^*)
= \sum_{\alpha\in \Sigma} \sigma(S_\alpha) \sigma(S_\alpha^*) =1
\end{equation*}
and similarly 
$U_{\sigma,\ttau}^* U_{\sigma,\ttau} =1$
so that 
$U_{\sigma,\ttau}$ is a unitary in $Q(H)$.
We also have for $a \in \AL$
\begin{align*}
U_{\sigma,\ttau} \ttau(a)
= &  \sum_{\alpha \in \Sigma} \sigma(S_\alpha) \ttau(S_\alpha^*)\ttau(a) 
=   \sum_{\alpha \in \Sigma} \sigma(S_\alpha) \ttau(S_\alpha^* a S_\alpha)\ttau(S_\alpha^*) \\
= &  \sum_{\alpha \in \Sigma} \sigma(S_\alpha S_\alpha^* a S_\alpha)\ttau(S_\alpha^* )
=   \sum_{\alpha \in \Sigma} \sigma( a S_\alpha)\ttau(S_\alpha^* )
=  \ttau(a) U_{\sigma,\ttau}.
\end{align*}
\end{proof}
\begin{lemma}\label{lem:5.4}
For an extension $\sigma:\OL\longrightarrow Q(H)$, 
take trivial extensions
$\tau_i:\OL\longrightarrow B(H)$ such that 
$\sigma|_{\AL} = \tilde{\tau}_i|_{\AL}, i=1,2$.
Then there exists a unitary
$V \in B(H)$ such that 
\begin{equation}\label{eq:lem2.4}
\pi(V) \sigma(a) = \sigma(a) \pi(V), \quad a \in \AL
\quad \text{ and } \quad
U_{\sigma,\ttau_2} = U_{\sigma,\ttau_1} \phi_{\ttau_1}(\pi(V)) \pi(V)^*
\end{equation}
where
$\phi_{\ttau_1}(\pi(V)) $ is a unitary in $\Q(H)$ defined by
\begin{equation*}
\phi_{\ttau_1}(\pi(V)) = \sum_{\alpha \in \Sigma} \ttau_1(S_\alpha) \pi(V) \ttau_1(S_\alpha)^*.
\end{equation*}
\end{lemma}
\begin{proof}
By Voiculescu's theorem \cite{Voiculescu}, one may find a unitary
$V \in B(H)$ such that
$
\ttau_2(X) =\pi(V) \ttau_1(X) \pi(V)^*
$
for $X \in \OL$.
For $a \in \AL$, we have
$$
\pi(V) \sigma(a) 
= \pi(V) \ttau_1(a) 
= \ttau_2(a) \pi(V) = \sigma(a) \pi(V),
$$
so that $\pi(V)$ commutes with $\sigma(a)$ for all $a \in \AL.$
It then follows that 
\begin{align*}
U_{\sigma,\ttau_2} 
= & \sum_{\alpha \in \Sigma} \sigma(S_\alpha) \pi(V) \ttau_1(S_\alpha^*)\pi(V)^* \\
= & \sum_{\alpha \in \Sigma} \sigma(S_\alpha) \ttau_1(S_\alpha^* S_\alpha) \pi(V) \ttau_1(S_\alpha^*)\pi(V)^* \\
= & \sum_{\alpha \in \Sigma} \sigma(S_\alpha) \ttau_1(S_\alpha^*) \sum_{\beta\in \Sigma} \ttau_1(S_\beta) \pi(V) \ttau_1(S_\beta^*)\pi(V)^* \\
= & U_{\sigma,\ttau_1} \phi_{\ttau_1}(\pi(V))\pi(V)^*. 
\end{align*}
It is routine to check that $\phi_{\ttau_1}(\pi(V))$ is a unitary in $Q(H)$
by using the commutativity between 
$\pi(V)$ and $\ttau_1(S_\alpha^* S_\alpha)$.
\end{proof}
\begin{lemma}\label{lem:5.5}
For a Fredholm module  $(u,\rho)$  over $\AL$,
take a trivial extension $\tau_\rho:\OL\longrightarrow B(H)$ such that 
$\rho = \tau_\rho|_{\AL}$.
Then
$\phi_{\ttau_\rho}(u) = \sum_{\alpha \in \Sigma} \ttau_\rho(S_\alpha) u \ttau_\rho(S_\alpha^*)$
is a unitary in $Q(H)$ commuting with $\pi(\rho(\AL))$.
Hence the pair
$(\phi_{\ttau_\rho}(u) , \rho)$ 
gives rise to an element of $\K^0(\AL)$ and its class
$[(\phi_{\ttau_\rho}(u) , \rho)]$ in $\K^0(\AL)$
is independent of the choice of the extension 
$\tau_\rho: \OL\longrightarrow B(H)$
as long as 
$\tau_\rho|_{\AL} = \rho|_{\AL}.$ 
\end{lemma}
\begin{proof}
We will show that 
$\phi_{\ttau_\rho}(u)$ 
commutes with 
$\pi(\rho(a))$ for $a \in \AL$.
We have
\begin{align*}
\phi_{\ttau_\rho}(u) \pi(\rho(a))
=&  \sum_{\alpha \in \Sigma} \ttau_\rho(S_\alpha) u \ttau_\rho(S_\alpha^* a)
=  \sum_{\alpha \in \Sigma} \ttau_\rho(S_\alpha) u \ttau_\rho(S_\alpha^* a S_\alpha)  \ttau_\rho(S_\alpha^*) \\
=&  \sum_{\alpha \in \Sigma} \ttau_\rho(S_\alpha) u \pi(\rho(S_\alpha^* a S_\alpha))  \ttau_\rho(S_\alpha^*) 
=  \sum_{\alpha \in \Sigma} \ttau_\rho(S_\alpha) \pi(\rho(S_\alpha^* a S_\alpha))u  \ttau_\rho(S_\alpha^*) \\
=&  \sum_{\alpha \in \Sigma} \ttau_\rho(S_\alpha S_\alpha^* a S_\alpha)) u  \ttau_\rho(S_\alpha^*) 
=  \sum_{\alpha \in \Sigma} \ttau_\rho(a) \ttau_\rho(S_\alpha) u  \ttau_\rho(S_\alpha^*) \\
= & \pi(\rho(a)) \phi_{\ttau_\rho}(u).
\end{align*}
Take two trivial extensions
$\tau_\rho, \tau'_{\rho}:\OL\longrightarrow B(H)$ such that 
$\rho = \tau_\rho|_{\AL} =\tau'_\rho|_{\AL}.$
By \eqref{eq:indl}, 
it suffices to show that 
\begin{equation}\label{eq:lem5.51}
\ind_{\pi\circ\rho(E_i^l)}\phi_{\tau_\rho}(u) 
= \ind_{\pi\circ\rho(E_i^l)}\phi_{\tau'_\rho}(u).
\end{equation}
It is straightforward to see that  
$\ttau_\rho(S_\alpha) u \ttau_\rho(S_\alpha^*)$ 
commutes  with $\pi(\rho(E_i^l))$,
 and the equality 
$$
\ind_{\pi\circ\rho(E_i^l)}\phi_{\tau_\rho}(u) 
= \sum_{\alpha \in \Sigma} 
\ind_{\pi\circ\rho(E_i^l)} \ttau_\rho(S_\alpha) u \ttau_\rho(S_\alpha^*)
$$
holds. 
One then has 
\begin{align*}
\ind_{\pi\circ\rho(E_i^l)} \ttau_\rho(S_\alpha) u \ttau_\rho(S_\alpha^*)
= & \ind_{\pi\circ\rho(E_i^l S_\alpha S_\alpha^* )} \ttau_\rho(E_i^l S_\alpha) u \ttau_\rho(S_\alpha^* E_i^l) \\
= & \ind_{\pi\circ\rho( S_\alpha^* E_i^l S_\alpha )} \ttau_\rho( S_\alpha^* E_i^l S_\alpha) u \ttau_\rho(S_\alpha^* E_i^l  S_\alpha) \\
= & \ind_{\pi\circ\rho( S_\alpha^* E_i^l S_\alpha )} \ttau'_\rho( S_\alpha^* E_i^l S_\alpha) u \ttau'_\rho(S_\alpha^* E_i^l  S_\alpha) \\
= & \ind_{\pi\circ\rho(E_i^l)} \ttau'_\rho( S_\alpha) u \ttau'_\rho(S_\alpha^* ),
\end{align*}
proving the equality \eqref{eq:lem5.51}.
\end{proof}
\begin{corollary}\label{cor:5.6}
For an extension $\sigma:\OL\longrightarrow Q(H)$, 
take  trivial extensions
$\tau, \tau':\OL\longrightarrow B(H)$
satisfying  
$\sigma|_{\AL} = \ttau|_{\AL}= \ttau'|_{\AL}$.
Then one may find a unitary
$V \in B(H)$ such that 
$
(\pi(V), \tau|_{\AL}) \in \K^0(\AL)
$
and the equality
\begin{equation}\label{eq:coro5.62}
\ind[(U_{\sigma,\ttau}, \tau|_{\AL})] -\ind[( U_{\sigma,\ttau'}, \tau'|_{\AL})]
= 
\ind[(\pi(V), \tau|_{\AL})] -\ind[( \phi_{\ttau}(\pi(V)), \tau|_{\AL})]
\end{equation}
holds.
\end{corollary}
\begin{proof}
Take a unitary $V \in B(H)$ satisfying \eqref{eq:lem2.4}.
We then have
$U_{\sigma,\ttau'} = U_{\sigma,\ttau}\phi_{\ttau}(\pi(V)) \pi(V)^*$
so that 
\begin{equation*}
\ind[(U_{\sigma,\ttau'}, \tau|_{\AL})]
= \ind[( U_{\sigma,\ttau}, \tau|_{\AL})]
+ \ind[( \phi_{\ttau}(\pi(V)), \tau|_{\AL})]
+ \ind[(\pi(V)^*, \tau|_{\AL})]. 
\end{equation*}
As 
$\ind[(U_{\sigma,\ttau'}, \tau|_{\AL})]
=\ind[(U_{\sigma,\ttau'}, \tau'|_{\AL})],
$ 
we get the equality \eqref{eq:coro5.62}.
\end{proof}
Define the subgroup
$\Z_0^{m(l)}$ of $\Z^{m(l)}$ by
$$
\Z_0^{m(l)} = \{ (n_i^l)_{i=1}^{m(l)}\in \Z^{m(l)} \mid \sum_{i=1}^{m(l)} n_i^l =0\}.
$$
As 
$I_{l,l+1}\Z_0^{m(l+1)} \subset \Z_0^{m(l)}, \, l \in \Zp$,  
we have a projective system 
$\{I_{l,l+1}:\Z_0^{m(l+1)} \longrightarrow  \Z_0^{m(l)}, \, l \in \Zp \}$
of abelian groups.
Define the subgroup $\Z_{I,0} $ of $\Z_I$ by
the abelian group 
$
\varprojlim \{I_{l,l+1}: 
\Z_0^{m(l+1)} \longrightarrow  \Z_0^{m(l)} \}
$
of the projective limit, so that 
$$
\Z_{I,0} = \{ (n^l)_{l\in \Zp} \in \prod_{l=0}^\infty \Z_0^{m(l)} \mid I_{l,l+1}n^{l+1} = n^l \}.
$$
The matrices 
$I_{l,l+1}, l \in \Zp$ act on $\Z_{I,0}$ by
$(n^l)_{l\in \Zp} \in \Z_{I,0} \longrightarrow (I_{l,l+1} n^{l+1})_{l\in \Zp}$
 as the identity denoted by $I$.
Define the subgroup $\K^0(\AL)_0$ of $\K^0(\AL)$ by
\begin{align*}
\K^0(\AL)_0 
= \{&  [(\pi(V), \rho)] \in \K^0(\AL) \mid
V \in U(B(H)) \},
\end{align*}
where $U(B(H))$ denotes the group of unitaries in $B(H)$.
\begin{lemma}\label{lem:16p}
The correspondence
$$
\Ind_l: [(u,\rho)] \in \K^0(\AL) 
\longrightarrow
(\ind_{\pi\circ\rho(E_i^l)}u)_{i=1}^{m(l)} \in \Z^{m(l)}
\quad \text{for } l \in \Zp 
$$ 
yields an isomorphism
$
\Ind : \K^0(\AL)_0 \longrightarrow \Z_{I,0}.
$
\end{lemma}
\begin{proof}
For $V \in U(B(H))$, we have
\begin{equation*}
\sum_{i=1}^{m(l)} \ind_{\pi\circ\rho(E_i^l)} \pi(V)
=\ind_{ \sum_{i=1}^{m(l)} \pi\circ\rho(E_i^l)} \pi(V)
=\ind_{\pi\circ\rho(1)} \pi(V)
=0
\end{equation*}
so that
$(\ind_{\pi\circ\rho(E_i^l)} \pi(V))_{i=1}^{m(l)} \in \Z_0^{m(l)}$.
Hence 
$\Ind( [(u,\rho)]) \in \Z_{I,0}$ for $[(u,\rho)]\in \K^0(\AL)_0.$
Conversely,
for any $(n^l)_{l \in \Zp}\in \Z_{I,0}$ with
$n^l =(n_i^l)_{i=1}^{m(l)} \in \Z_0^{m(l)}, l \in \Zp,$
 one may find a Fredholm module $(u,\rho) $ over $\AL$ such that 
 $\Ind[(u,\rho)] = (n^l)_{l\in \Zp}$.
 As
$$
0 = \sum_{i=1}^{m(l)} \ind_{\pi\circ\rho(E_i^l)}u = \ind_H(u),
$$ 
there exists a unitary $V \in B(H)$ such that 
$\pi(V) =u$.
Hence we have
$\Ind[(\pi(V),\rho)] = (n^l)_{l \in \Zp} \in \Z_{I,0}$,
proving
$\Ind(\K^0(\AL)_0) = \Z_{I,0}.$
\end{proof}
Let  $(A_{l,l+1}, I_{l,l+1})_{l\in \Zp}$ 
be the structure matrices for the $\lambda$-graph system $\L$.
Put
$A_{l,l+1}^\L(i,j) = \sum_{\alpha \in \Sigma} A_{l,l+1}(i,\alpha,j)$
for $i=1,\dots, m(l), \, j=1,\dots, m(l+1)$, which satisfies
the relation
\begin{equation}\label{eq:commuteIA}
I_{l,l+1} A^\L_{l+1,l+2} =A^\L_{l,l+1} I_{l+1,l+2},
\qquad
l \in \Zp.
\end{equation}
Let
$A_\L:\Z_I \longrightarrow \Z_I$ 
be the endomorphism on the group
$\Z_I$ defined by
$$
A_\L((x_l)_{l\in \Zp}) 
= (A^\L_{l,l+1}x_{l+1})_{l\in \Zp}.
$$
By the identity \eqref{eq:commuteIA}, 
we know that 
$ (A^\L_{l,l+1}x_{l+1})_{l\in \Zp} $ belongs to $\Z_I$ 
for 
$(x_l)_{l\in \Zp} \in \Z_I$, 
so that
$A_\L:\Z_I \longrightarrow \Z_I$ 
yields an endomorphism on $\Z_I$.
\begin{lemma}\label{lem:5.7}
The map
$\phi: [(\pi(V),\rho)] \in \K^0(\AL)_0 \longrightarrow 
[(\phi_{\ttau_\rho}(\pi(V)), \rho)] \in \K^0(\AL)
$
gives rise to the commutative diagram
\begin{equation}\label{eq:CDphiAL}
\begin{CD}
\K^0(\AL)_0 @>{\phi}>> \K^0(\AL) \\
@V{\Ind}VV     @VV{\Ind}V \\ 
\Z_{I,0} @>{A_\L}>> \Z_I.
\end{CD}
\end{equation}
\end{lemma}
 \begin{proof}
We have
 \begin{align*}
 \ind_{\pi\circ\rho(E_i^l)}\phi_{\ttau_\rho}(\pi(V))
 = & \sum_{\alpha \in \Sigma}  \ind_{\pi\circ\rho(E_i^l)} \ttau(S_\alpha) \pi(V) \ttau(S_\alpha^*) \\
 = & \sum_{\alpha \in \Sigma}  \ind_{\pi\circ\rho(E_i^l)} \pi(\tau(S_\alpha) V \tau(S_\alpha^*))\\
 = & \sum_{\alpha \in \Sigma}  \ind_{\pi\circ\rho(S_\alpha^* E_i^l S_\alpha)} \pi(\tau(S_\alpha^* E_i^l S_\alpha) V \tau(S_\alpha^* E_i^l S_\alpha))\\
 = & \sum_{\alpha \in \Sigma}  \sum_{j=1}^{m(l+1)} A_{l,l+1}(i,\alpha,j)
     \ind_{\pi\circ\rho( E_j^{l+1})} \pi(\tau( E_j^{l+1}) V \tau( E_j^{l+1})) \\
 = &  \sum_{j=1}^{m(l+1)}  A_{l,l+1}^\L(i,j)
     \ind_{\pi\circ\rho( E_j^{l+1})} \pi(V)
  \end{align*}
  so that we have
  $A_\L \circ \Ind = \Ind \circ \phi: \K^0(\AL)_0 \longrightarrow \Z_I.$
 \end{proof}
For an extension $\sigma:\OL\longrightarrow Q(H)$, 
take a trivial extension
$\tau:\OL\longrightarrow B(H)$ satisfying 
$\sigma|_{\AL} = \pi\circ \tau|_{\AL}$
and consider the unitary
$U_{\sigma,\ttau}= \sum_{\alpha \in \Sigma} \sigma(S_\alpha) \ttau(S_\alpha^*) \in Q(H)$.
By Lemma \ref{lem:5.3}, the pair 
$(U_{\sigma,\ttau}, \tau|_{\AL})$ defines an element of $\K^0(\AL)$.
Put the Fredholm module 
\begin{equation*}
d(\sigma,\tau) = (U_{\sigma,\ttau}, \tau|_{\AL})
\end{equation*}
over $\AL$
so that the class
$[d(\sigma,\tau)]$ defines an element of $\K^0(\AL)$.
Corollary \ref{cor:5.6} together with
Lemma \ref{lem:5.7}
says that 
for  trivial extensions
$\tau, \tau':\OL\longrightarrow B(H)$ such that 
$\sigma|_{\AL} = \ttau|_{\AL}= \ttau'|_{\AL}$,
we have
\begin{align*}
\Ind[d(\sigma,\tau)] - \Ind[d(\sigma,\tau')]
= &(I - A_\L)  \ind[(\pi(V),\tau|_{\AL})].  
\end{align*}
As $\ind[(\pi(V),\tau|_{\AL})] \in \Z_{I,0}$,
the class
$$
[\Ind[d(\sigma,\tau)]] \in \Z_I/(I-A_\L)\Z_{I,0}
$$
is independent of the choice of $\tau$ 
as long as $\pi\circ\tau|_{\AL} =\sigma|_{\AL}.$
\begin{lemma}\label{lem:24p}
For an extension $\sigma_1: \OL \longrightarrow Q(H)$,
take a trivial extension
$\tau_1:\OL\longrightarrow B(H)$ 
such that 
$\sigma_1|_{\AL} = \pi\circ \tau_1|_{\AL}$.
Let
$\sigma_2:\OL \longrightarrow Q(H)$
be an extension strongly equivalent to 
$\sigma_1$.
Then there exists 
a trivial extension
$\tau_2:\OL\longrightarrow B(H)$ such that 
$\sigma_2|_{\AL} = \pi\circ \tau_2|_{\AL}$
and
$$
[\Ind[d(\sigma_1,\tau_1)]] =
[\Ind[d(\sigma_2,\tau_2)]]
\text{ in }
\Z_I/(I - A_\L)\Z_{I,0}.
$$
\end{lemma}
\begin{proof}
Since $\sigma_2$ is strongly equivalent to
$\sigma_1$,
one may take a unitary
$V \in B(H)$ such that 
$\sigma_2 = \Ad(\pi(V))\circ \sigma_1$.
Define a trivial extension
$\tau_2:\OL\longrightarrow B(H)$
by $\tau_2 = \Ad(V) \circ \tau_1$
satisfying 
$\sigma_2|_{\AL} = \pi\circ \tau_2|_{\AL}$. 
We then have
\begin{align*}
\Ind[d(\sigma_2,\tau_2)]
= & [(\ind_{\sigma_2(E_i^l)}
     U_{\sigma_2,\ttau_2})_{i=1}^{m(l)})_{l\in \Zp}] \\ 
= & [(\ind_{\sigma_2(E_i^l)}
     \sum_{\alpha \in \Sigma} 
     \sigma_2(S_\alpha)\ttau_2(S_\alpha^*))_{i=1}^{m(l)})_{l\in \Zp}].
\end{align*}
Now we have
\begin{align*}
    \ind_{\sigma_2(E_i^l)}\sum_{\alpha \in \Sigma} \sigma_2(S_\alpha)\ttau_2(S_\alpha^*)
= & \ind_{\pi(V) \sigma_1(E_i^l)\pi(V^*)}
    \sum_{\alpha \in \Sigma} \pi(V)\sigma_1(S_\alpha)\ttau_1(S_\alpha^*)\pi(V)^* \\
= & \ind_{\sigma_1(E_i^l)}
    \sum_{\alpha \in \Sigma} \sigma_1(S_\alpha)\ttau_1(S_\alpha^*) 
=  \ind_{\sigma_1(E_i^l)}
    U_{\sigma_1, \ttau_1}
\end{align*}
so that $\Ind[d(\sigma_1,\tau_1)]
 =\Ind[d(\sigma_2,\tau_2)].$
\end{proof}
Therefore we have 
\begin{proposition}\label{prop:5.8}
For the strong euivalence class 
$[\sigma]_s \in \Exts(\OL)$
of an extension $\sigma:\OL\longrightarrow Q(H)$,
the class $[\Ind[d(\sigma,\tau)]]$ in the group
$\Z_I/(I-A_\L)\Z_{I,0}$
is independent of the choice of a trivial extension
$\tau:\OL\longrightarrow B(H)$ as long as
$\sigma|_{\AL} = \tau|_{\AL}$. 
\end{proposition}
We thus have a homomorphism
$$\Inds: \Exts(\OL) \longrightarrow \Z_I/(I-A_\L)\Z_{I,0}$$ 
defined by
$\Inds([\sigma]_s) = [\Ind[d(\sigma,\tau)]].$

\medskip
Take a trivial extension
$\tau:\OL\longrightarrow B(H)$
and a unitary $u_m \in Q(H)$ of Fredholm index $m \in \Z$
such that 
$\pi(\tau(a)) u_m = u_m\pi(\tau(a))$ for $a \in \AL$.
Define an extension
$\sigma_m:\OL\longrightarrow B(H)$
by $\sigma_m = \Ad(u_m)\circ( \pi\circ\tau)$
so that the class $[\sigma_m]_s$ in $\Exts(\OL)$
is defined for each $m \in \Z$.
We then have a homomorphism
$\iota_\L: m\in \Z \longrightarrow [\sigma_m]_s \in \Exts(\OL)$
such that the sequence
\begin{equation}
\Z \overset{\iota_\L}{\longrightarrow}\Exts(\OL) \overset{q}{\longrightarrow} \Extw(\OL)   
\end{equation}
is exact at the middle, where 
$q:\Exts(\OL)\longrightarrow \Extw(\OL)$ is the natural quotient map (cf. \cite{PP}).

We will introduce a homomorphism
$\hat{\iota}_\L:\Z\longrightarrow \Z_I/(I-A_\L)\Z_{I,0}$
in the following way.
For $m \in \Z$, take an element 
$(n^l)_{l\in \Zp}\in \Z_I$
such that 
$n^l =(n^l_i)_{i=1}^{m(l)} \in \Z^{m(l)}$
and
$m =\sum_{i=1}^{m(l)} n_i^l$ for each $l \in \Zp$.
One may take such a sequence 
as $n^l =(m,0,\dots,0)\in \Z^{m(l)}$.
We then define
\begin{equation*}
\hat{\iota}_\L(m) = [(I- A_\L)[(n^l)_{\in \Zp}]] \in \Z_I/ (I- A_\L)\Z_{I,0}.
\end{equation*}
Let $(n^{\prime l})_{l\in \Zp} \in \Z_I$
be another sequence such that 
$m = \sum_{i=1}^{m(l)} n_i^{\prime l}.$
As
$\sum_{i=1}^{m(l)} (n_i^l - n_i^{\prime l}) = m -m =0,$
we have
$(n^{l})_{l\in \Zp}- (n^{\prime l})_{l\in \Zp} \in \Z_{I,0}$
so that 
$[(I- A_\L)[(n^l)_{\in \Zp}]]
=[(I- A_\L)[(n^{\prime l})_{\in \Zp}]] \in \Z_I/ (I- A_\L)\Z_{I,0}.
$
This shows that 
$\hat{\iota}_\L(m)$ is independent of the choice of $(n^l)_{l\in \Zp}\in \Z_I$
as long as $m = \sum_{i=1}^{m(l)} n_i^l.$
\begin{lemma}\label{lem:29p}
The diagram
\begin{equation}\label{eq:CDextsz}
\begin{CD}
\Z @>{\iota_\L}>> \Exts(\OL)  \\
\parallel @. @VV{\Inds}V      \\ 
\Z @>{\hat{\iota}_\L}>> \Z_I /(I-A_\L)\Z_{I,0} 
\end{CD}
\end{equation}
commutes.
\end{lemma}
 \begin{proof}
 Take a trivial extension
$\tau:\OL\longrightarrow B(H)$
and a unitary $u_m \in Q(H)$ of Fredholm index $m \in \Z$
such that 
$\pi(\tau(a)) u_m = u_m\pi(\tau(a))$ for $a \in \AL$.
The extension
$\sigma_m:\OL\longrightarrow B(H)$
is defined by $\sigma_m = \Ad(u_m)\circ( \pi\circ\tau)$.
Put
$k_i^l 
= \ind_{\sigma_m(E_i^l)}u_m 
=\ind_{\pi(\tau(E_i^l))}u_m.
$
As
$\ind(u_m) = m$,
we have
$\sum_{i=1}^{m(l)} k_i^l =m$ for each $l \in \Zp.$
Now we have
\begin{equation*}
\Ind_s([\sigma_m]_s) 
=[((\ind_{\sigma_m(E_i^l)} U_{\sigma_m,\ttau})_{i=1}^{m(l)})_{l\in \Zp}]
=[(( \ind_{\pi(\tau(E_i^l))} 
\sum_{\alpha\in \Sigma}\sigma_m(S_\alpha)\pi\circ\tau(S_\alpha^*)))_{l\in \Zp}].
\end{equation*} 
Since we have
\begin{align*}
   &  \ind_{\pi(\tau(E_i^l))} 
    \sum_{\alpha\in \Sigma}\sigma_m(S_\alpha)\pi(\tau(S_\alpha^*)) \\
= & \ind_{\pi(\tau(E_i^l))} 
    u_m ( \sum_{\alpha\in \Sigma}\pi(\tau(S_\alpha)) u_m^* 
    \pi(\tau(S_\alpha^*))) \\
= & \ind_{\pi(\tau(E_i^l))} u_m 
  + \sum_{\alpha\in \Sigma} \ind_{\ttau(E_i^l)}\ttau(S_\alpha) u_m^* \ttau(S_\alpha^*) \\
= & k_i^l 
  + \sum_{\alpha\in \Sigma} \ind_{\ttau(S_\alpha^*E_i^l S_\alpha)}u_m^*  \\
= &  k_i^l 
  + \sum_{\alpha\in \Sigma} \sum_{j=1}^{m(l)} A_{l,l+1}(i,\alpha, j)
    \ind_{\ttau(E_j^{l+1})}u_m^*  \\
= &   k_i^l - \sum_{j=1}^{m(l)} A^\L_{l,l+1}(i, j) k_j^{l+1}
=   (I - A^\L)[(k_j^{l+1})_{j=1}^{m(l+1)}],
\end{align*}
we get
\begin{equation*}
\Ind_s([\sigma_m]_s) =(I - A^\L)[(k_j^{l+1})_{j=1}^{m(l+1)}].
\end{equation*} 
As
$\hat{\iota}_\L(m) = [(I-A_\L)[(k^l)_{l\in\Zp}]$,
we conclude  
$\Ind_s([\sigma_m]_s)=\hat{\iota}_\L(m),$ 
proving
$\Ind_s(\iota_\L(m)) =\hat{\iota}_\L(m).$
\end{proof}
Define a homomorphism 
$s_\L: \Ker(I-A_\L: \Z_I \longrightarrow \Z_I) \longrightarrow \Z$
by setting
$s_\L((n^l)_{l\in \Zp}) = \sum_{i=1}^{m(l)} n_i^l$ 
which is independent of $l\in \Zp.$
\begin{lemma}\label{lem:slambda}
We have a cyclic six-term exact sequence
\begin{equation*}
\begin{CD}
 \Ker(I-A_\L: \Z_{I,0} \longrightarrow \Z_I)
  @>{}>> \Ker(I-A_\L: \Z_I \longrightarrow \Z_I) 
  @>{s_\L}>> \Z \\
@AAA  @.   @VV{\hat{\iota}_\L}V  \\
0 @<<< \Z_I/(I-A_\L)\Z_I  
@<{}<< \Z_I/(I-A_\L)\Z_{I,0}.
\end{CD}
\end{equation*}
\end{lemma}
\begin{proof}
It suffices to show that
$\Ker(\hat{\iota}_\L) 
=s_\L(\Ker(I-A_\L: \Z_I \longrightarrow \Z_I)).$
For $m \in \Ker(\hat{\iota}_\L)$ in $\Z$,
we have 
$\hat{\iota}_\L(m) = (I - A_\L)[(k^l)_{l\in \Zp}]$ 
where 
 $k^l = (k_i^l)_{i=1}^{m(l)}, m =\sum_{i=1}^{m(l)}k_i^l.$
Since
$\hat{\iota}_\L(m) \in (I-A_\L)\Z_{I,0}$,
one may find $(n^l)_{l\in \Zp} \in \Z_{I,0}$
such that 
$(I - A_\L)[(k^l)_{l\in \Zp}] 
= (I - A_\L)[(n^l)_{l\in \Zp}].$
We then have
$(I - A_\L)[(k^l)_{l\in \Zp}]
- (I - A_\L)[(n^l)_{l\in \Zp}] =0$
so that 
$(k^l - n^l)_{l\in \Zp} \in \Ker(I- A_\L)$
and
$m = \sum_{i=1}^{m(l)} k_i^l = \sum_{i=1}^{m(l)} (k_i^l- n_i^l)$.
This shows that
$$
m = s_\L((k^l - n^l)_{l\in \Zp}) 
\in s_\L(\Ker(I-A_\L: \Z_I \longrightarrow \Z_I)).
$$
Conversely, 
for $(n^l)_{l\in \Zp} \in \Ker(I-A_\L: \Z_I \longrightarrow \Z_I))$,
we have
\begin{equation*}
\hat{\iota}_\L(s_\L((n^l)_{l\in \Zp} ) )
=\hat{\iota}_\L(\sum_{i=1}^{m(l)}n_i^l) = (I-A_\L)((n^l)_{l\in \Zp}) =0.
\hspace{35mm}
\qed
\end{equation*}
\renewcommand{\qed}{}
\end{proof}
Following Higson\cite{Higson} and Higson--Roe \cite{HR}, 
for  a separable unital nuclear $C^*$-algebra $\A$
the reduced $\K$-homology groups 
$\widetilde{\K}^0(\A), \widetilde{\K}^1(\A) $
and the  unreduced $\K$-homology groups 
${\K}^0(\A), {\K}^1(\A)$
are identified with their extension groups such as 
$$
\widetilde{\K}^0(\A)= \Extsz(\A), \quad
\widetilde{\K}^1(\A)= \Ext_s(\A), \quad
{\K}^0(\A)= \Extwz(\A), \quad
{\K}^1(\A)= \Ext_w(\A), \quad
$$
respectively.
The groups 
$\Exts(\A)$ and $\Extw(\A)$ are written
as
$\Extso(\A)$ and $\Extwo(\A),$
respectively.
The isomorphisms
$\Inds: \Exts(\A)\ \longrightarrow \Z_I/(I -A_\L)\Z_{I,0}$
and
$\Indw: \Extw(\A)\ \longrightarrow \Z_I/(I -A_\L)\Z_{I}$
are written as
$\Indso$ and $\Indwo$, respectively.
A general theory of 
$\K$-homology groups for a separable unital nuclear $C^*$-algebr $\A$
says that the following $\K$-homology long exact sequence holds:
\begin{equation}\label{eq:Khomo}
0
\longrightarrow \widetilde{{\K}}^0(\A)
\longrightarrow {\K}^0(\A)
\overset{\iota^*_{\mathbb{C}}}{\longrightarrow}{\K}^0(\mathbb{C})=\Z 
\overset{\iota}{\longrightarrow}  \widetilde{\K}^1(\A)
\longrightarrow \K^1(\A)
\longrightarrow 0.
\end{equation}
By \cite{MaKtheory2001}, 
we have already known that 
\begin{equation}
\Extwo(\OL) =\Z_I/(I -A_\L)\Z_I, \qquad 
\Extwz(\OL) =\Ker(I-A_\L: \Z_I \longrightarrow \Z_I).
\end{equation}
The homomorphism 
$\iota_{\mathbb{C}}^*: \K^0(\OL) \longrightarrow \K^0(\mathbb{C})$
in the middle of \eqref{eq:Khomo} for $\A=\OL$
is defined by the natural unital inclusion map
$\iota_{\mathbb{C}}:\C\hookrightarrow \OL$.
As the number 
$\sum_{i=1}^{m(l+1)} n_i^l$ 
for 
$n^l = (n_i^l)_{i=1}^{m(l)} \in \Ker(I-A_\L: \Z_I \longrightarrow \Z_I) $
does not depend on the choice of
$l \in \Zp,$ 
the homomorphism 
$\iota_{\mathbb{C}}^*: \K^0(\OL) \longrightarrow \K^0(\mathbb{C}) =\Z$
satisfies 
$\iota_{\mathbb{C}}^*( (n^l)_{l\in \Zp}) =  \sum_{i=1}^{m(l+1)} n_i^l$
which does not depend on $l\in \Zp$.
Since
$\widetilde{\K}^0(\OL) 
= \Ker( \iota_{\mathbb{C}}^*: \K^0(\OL) \longrightarrow \K^0(\mathbb{C})),
$
we know that 
\begin{equation}\label{eq:Ktilde1}
\widetilde{\K}^0(\OL) 
=\Ker( I - A_\L: \Z_{I,0}\longrightarrow \Z_I ).
\end{equation}
The cyclic six-term exact sequence \eqref{eq:Khomo} says the following lemma.
 \begin{lemma}\label{lem:longexact}
The following diagram is commutative:
\begin{equation}\label{eq:long6termexacts}
\begin{CD}
 0 @. 0 @. 0  \\
 @VVV   @VVV @VVV  \\
 \widetilde{\K}^0(\OL) @= \Extsz(\OL) @>>> 
 \Ker(I-A_\L: \Z_{I,0} \longrightarrow \Z_I) \\
 @V{}VV   @VV{}V  @VV{}V \\
 \K^0(\OL) @= \Extwz(\OL) @>>> 
 \Ker(I-A_\L: \Z_{I} \longrightarrow \Z_I) \\
 @V{\iota_{\mathbb{C}}^*}VV   @VV{s_\L}V  @VV{s_\L}V \\
 \K^0(\mathbb{C})  @= \Z    @= \Z \\
 @V{\iota^*}VV  @V{\iota^*}VV   @VV{\hat{\iota}_{\L}}V \\
 \widetilde{\K}^1(\OL) @= \Extso(\OL)  @>{\Indso}>>   \Z_I/(I - A_\L)\Z_{I,0}  \\
 @V{}VV @V{}VV   @VV{}V \\
  \K^1(\OL) @= \Extwo(\OL)  @>{\Indwo}>>\Z_I/ (I-A_\L)\Z_I  \\
 @VVV   @VVV @VVV \\
  0 @. 0 @. 0
 \end{CD}
 \end{equation}
 where
 $\hat{\iota}_\L: \Z \longrightarrow \Z_I/)I - A_\L)\Z_{I,0}$
 is defined by $\hat{\iota}_\L(m) = [(I - A_\L)[(n^l)_{l\in \Zp}]]$
 for $ m = \sum_{i=1}^{m(l)} n_i^l, \, l \in \Zp,$
 and
 $s_\L: \Ker(I- A_\L: \Z_I \longrightarrow \Z_I) \longrightarrow \Z$
 is defined by 
 $s_\L((n^l)_{l\in \Zp}) =  \sum_{i=1}^{m(l)} n_i^l$
 for $(n^l)_{l\in \Zp}\in \Ker(I - A_\L:\Z_I\longrightarrow \Z_I).$
\end{lemma}
\begin{corollary}
$\Indso: \Exts(\OL) \longrightarrow \Z_I/(I - A_\L)\Z_{I,0}$
is an isomorphism of abelian groups. 
\end{corollary}
\begin{proof}
By the commutative diagram \eqref{eq:long6termexacts},
we have a commutative diagram of short exact sequences:
\begin{equation*}
\begin{CD}
0 @>>> \Z/ s_\L(\Extwz(\OL)) @>{\iota}>> \Extso(\OL) @>>> \Extwo(\OL) @>>>0 \\
@.      @VV{}V   @V{\Indso}VV    @V{\Indwo}VV @. \\ 
0 @>>>  \Z/s_\L(\Ker(I - A_\L)) @>>> \Z_I/(I - A_\L)\Z_{I,0}
  @>>> \Z_I/(I - A_\L)\Z_I @>>>0. 
\end{CD}
\end{equation*}
Since the two vertical arrows 
$\Indwo: \Extwo(\OL) \longrightarrow \Z_I/(I - A_\L)\Z_I$ 
and 
$\Z/ s_\L(\Extwz(\OL)) \longrightarrow \Z/s_\L(\Ker(I - A_\L))$
are isomorphisms,
the five lemma says that the middle homomorphism
$\Indso: \Extso(\OL) \longrightarrow \Z_I/(I - A_\L)\Z_{I,0}$ 
is isomorphic.
 \end{proof}
We therefore obtain the following theorem.
\begin{theorem}[{Theorem \ref{thm:main1}}]\label{thm:main11}
Let $\L$ be a left-resolving $\lambda$-graph system over $\Sigma$.
There exist isomorphisms
\begin{align*}
\Indwo: \Extwo(\OL) & \longrightarrow \Z_I/(I - A_\L)\Z_I,\\
\Indso: \Extso(\OL) & \longrightarrow \Z_I/(I - A_\L)\Z_{I,0},\\
\Indwz: \Extwz(\OL) & \longrightarrow \Ker(I - A_\L:\Z_I\longrightarrow \Z_I),\\
\Indsz: \Extsz(\OL) & \longrightarrow \Ker(I - A_\L: \Z_{I,0}\longrightarrow \Z_I)
\end{align*} 
of abelian groups such that 
the $\K$-homology long exact sequence 
\eqref{eq:Khomo} 
is computed to be
the cyclic six-term exact sequence
\begin{equation*}
\begin{CD}
\Ker(I - A_\L: \Z_{I,0}\longrightarrow \Z_I) 
@>{}>> \Ker(I - A_\L: \Z_I\longrightarrow \Z_I) 
@>>> \Z \\
@AAA  @.   @VV{}V  \\
0 @<<< \Z_I/(I - A_\L)\Z_I  @<{}<< \Z_I/(I - A_\L)\Z_{I,0}. 
\end{CD}
\end{equation*}
\end{theorem}

\section{Examples}
\subsection{Markov coded systems}
Let $G = (V, E)$ be a finite directed graph with 
vertex set $V = \{v_1, \dots, v_N\}$ and edge set 
$E = \{e_1, \dots, e_{N_1}\}.$
We assume that the graph $G$ is essential,
which means that every vertex has at least one incoming edges 
and at least one outgoing edges.
Let $b, c $ be two letters.
Consider the set
\begin{equation}
\calC_G := \{ \overbrace{b\cdots b}^{n} \overbrace{c\cdots c}^{m} e_k \mid
k=1,\dots, N_1, \, n\le m, \, n, m \in \N\}  
\end{equation}
which is called a code for $G$.
Define the map $r: \calC_G \rightarrow E$ by 
$r(b\cdots b c\cdots c e_k) =e_k$.
Put $\Sigma = \{b, c, e_n \mid n=1,\dots,N_1\}.$ 
Let $\Omega_{(\calC_G, r)}$
be a shift-invariant set defined by setting 
\begin{align*}
\Omega_{(C_G,r)}
:= & \{ (\omega_i)_{i \in \Z} \in \Sigma^\Z \mid
\text{there exist } \cdots < k_{-1} < k_0 < k_1 <\cdots \text{ in } \Z;\\
& \hspace{2cm} \omega_{[k_i, k_{i+1})} \in \calC_G,  
t(r(\omega_{[k_i, k_{i+1})})) = s(r(\omega_{[k_i, k_{i+1})})), i \in \Z\}
\end{align*}
where $\omega_{[k_i, k_{i+1})} = \omega_{k_i}\cdots \omega_{k_{i+1}-1}$
and
$t(e), s(e)$ for an edge $e\in E$ denote the target vertex and the source vertex,
respectively.  
The set $\Omega_{(C_G, r)}$ is shift-invariant but not necessarily closed  
in $\Sigma^\Z$.
The closure 
$\overline{\Omega_{(C_G, r)}}$ is a shift space of a subshift.
The subshift is called the Markov coded system and written $S_G$
(\cite{MaActaSci2014}).  
It is a normal subshift in the sense of \cite{MaDocMath2023}
and not any of  topological Markov shifts for every finite directed graph $G$.
There is a $\lambda$-graph system written   
$\L^{S_G}$ canonically constructed from the Markov coded system $S_G$.
It presents the subshift $S_G$ and is minimal in the sense of 
\cite{MaDocMath2023}.
The $C^*$-algebra ${\mathcal{O}}_{\L^{S_G}}$
associated with the $\lambda$-graph system $\L^{S_G}$
is written as $\OSG$ in \cite{MaActaSci2014}.
It is shown in \cite{MaActaSci2014} that the algebra
$\OSG$ is simple purely infinite if the transition matrix $A$
of the directed graph $G$ is aperiodic, 
and the $\K$-theory groups and the weak extension groups are
such as 
\begin{gather*}
\K_0(\OSG)\cong  \Z^N/A\Z^N \oplus \Z^N, \qquad
\K_1(\OSG) \cong  \Ker(A) \text{ in } \Z^N, \\
\Extwz(\OSG)  \cong  (\Ker(A) \text{ in } \Z^N) \oplus \Z^N, \quad 
\Extwo(\OSG) \cong \Z^N/A\Z^N.
\end{gather*}
We note that 
in \cite{MaActaSci2014}
 the weak extension groups 
$\Extwo(\OSG), \Extwz(\OSG)
$
are written as $\operatorname{Ext}^0(\OSG),\operatorname{Ext}^1(\OSG),$
respectively.
In what follows, 
we will compute the strong extension groups 
$
\Extso(\OSG), \Extsz(\OSG).
$

Let us denote by $1_N$ and $0_N$ the identity matrix of size $N$
and $0$ matrix of size $N$.  
The canonical $\lambda$-graph system $\L^{S_G}$ and its 
transition matrices 
$(A_{l,l+1}^{\L^{S_G}}, I_{l,l+1}^{\L^{S_G}})$,
written as 
$(M_{l,l+1}, I_{l,l+1})$ in \cite{MaActaSci2014}
for $\L^{S_G}$ was concretely computed such as 
\begin{equation*}   
M_{l,l+1}
=\addtocounter{MaxMatrixCols}{1}
\begin{bmatrix}
1_N    &0_N   &\cdots &\cdots&\cdots&\cdots &\cdots &0_N    & 0_N  & 1_N \\
0_N    &1_N   &\ddots &      &      &       &\iddots&0_N    & 1_N  & 0_N  \\ 
\vdots &\ddots&\ddots &\ddots&      &\iddots&\iddots&\iddots&\iddots& \vdots \\
0_N    &\cdots&0_N    &1_N   & 0_N  & 0_N   &1_N    &0_N    &\cdots& 0_N\\
A^t    &\cdots&\cdots & A^t  & A^t  & 0_N   &0_N    &\cdots & 0_N  & 0_N \\
0_N    &\cdots&\cdots & 0_N  & 0_N  & 1_N   &\ddots & \ddots&\vdots&\vdots\\
\vdots &      &       &\vdots&\vdots&\ddots&\ddots&0_N   &0_N&0_N\\
0_N    &\cdots&\cdots & 0_N  & 0_N  & \cdots&0_N    & 1_N   & 1_N& 1_N
\end{bmatrix}
\end{equation*}
\begin{equation*}   
I_{l,l+1}
=\addtocounter{MaxMatrixCols}{1}
\begin{bmatrix}
1_N    &1_N   &0_N    &\cdots&\cdots&\cdots &\cdots &\cdots &0_N   & 0_N \\
0_N    &0_N   &1_N    &\ddots&      &       &       &       &\vdots&\vdots\\ 
\vdots &\ddots&\ddots &\ddots&\ddots&       &       &       &\vdots& \vdots \\
\vdots &      &\ddots &\ddots&\ddots&\ddots &       &       &\vdots&\vdots \\
\vdots &      &       &\ddots&\ddots&\ddots &\ddots &       &\vdots&\vdots \\
\vdots &      &       &      &\ddots&\ddots &\ddots & \ddots&\vdots&\vdots\\
\vdots &      &       &      &      &\ddots &0_N    & 1_N   &0_N   &0_N\\
0_N    &\cdots&\cdots &\cdots&\cdots& \cdots&0_N    & 0_N   & 1_N  & 1_N
\end{bmatrix}
\end{equation*}
for $3\le l \in \N$,
where both $M_{l,l+1}, I_{l,l+1}$ are 
$2(l+1) \times 2(l+2)$-block matrices whose entries are $N\times N$-matrices,
so they are
$m(l) \times m(l+1)$ matrix with
$m(l) = 2(l+1)N$.
Hence we have
\begin{equation*}   
M_{l,l+1}-I_{l,l+1}
=\addtocounter{MaxMatrixCols}{1}
\begin{bmatrix}
0_N    &-1_N  &\cdots &\cdots&\cdots&\cdots &\cdots &0_N    & 0_N  & 1_N \\
0_N    &1_N   &\ddots &      &      &       &\iddots&0_N    & 1_N  & 0_N  \\ 
\vdots &\ddots&\ddots &\ddots&      &\iddots&\iddots&\iddots&\iddots& \vdots \\
0_N    &\cdots&0_N    &1_N   &-1_N  & 0_N   &1_N    &0_N    &\cdots& 0_N\\
A^t    &\cdots&\cdots & A^t  & A^t  &-1_N   &0_N    &\cdots & 0_N  & 0_N \\
0_N    &\cdots&\cdots & 0_N  & 0_N  & 1_N   &\ddots & \ddots&\vdots&\vdots\\
\vdots &      &       &\vdots&\vdots&\ddots&\ddots  &-1_N   &0_N   &0_N\\
0_N    &\cdots&\cdots & 0_N  & 0_N  & \cdots&0_N    & 1_N   & 0_N  & 0_N
\end{bmatrix}.
\end{equation*}
\begin{lemma}
The homomorphism
$s_{\L^{S_G}}: \Ker(I - A_{\L^{S_G}}: \Z_I \rightarrow \Z_I)\rightarrow \Z$
in the upper right horizontal arrow in Lemma \ref{lem:slambda} for $\L = \L^{S_G}$
is surjective.
\end{lemma}
\begin{proof}
Since the map 
$s_{\L^{S_G}}: \Ker(I - A_{\L^{S_G}}: \Z_I \rightarrow \Z_I)\rightarrow \Z$
is defined by 
$s_{\L^{S_G}}((n^l)_{l \in \Zp}) = \sum_{i=1}^{m(l)}$
that is independent of $l \in \Zp$, 
we choose $l =3$
and consider 
the kernel $\Ker(M_{3,4} -I_{3,4})$.
The matrix $M_{3,4} -I_{3,4}$ is of the form:
\begin{equation*}   
M_{3,4}-I_{3,4}
=\addtocounter{MaxMatrixCols}{1}
\begin{bmatrix}
0_N    &-1_N  &0_N    &0_N   & 0_N  & 0_N   & 0_N   &0_N    & 0_N  & 1_N \\
0_N    &1_N   &-1_N   &0_N   & 0_N  & 0_N   & 0_N   &0_N    & 1_N  & 0_N  \\ 
0_N    &0_N   &1_N    &-1_N  & 0_N  & 0_N   & 0_N   &1_N    &0_N   & 0_N   \\
0_N    &0_N   &0_N    &1_N   &-1_N  & 0_N   &1_N    &0_N    &0_N   & 0_N\\
A^t    &A^t   &A^t    & A^t  & A^t  &-1_N   &0_N    &0_N    & 0_N  & 0_N \\
0_N    &0_N   &0_N    & 0_N  & 0_N  & 1_N   &-1_N   &0_N    &0_N   &0_N   \\
0_N    &0_N   &0_N    & 0_N  &0_N   &0_N    & 1_N   &-1_N   &0_N   &0_N\\
0_N    &0_N   &0_N    & 0_N  & 0_N  &0_N    &0_N    & 1_N   & 0_N  & 0_N
\end{bmatrix}.
\end{equation*}
It is easy to see that $[x_i]_{i=1}^{10}$ with $x_i \in \Z^N, i=1,\dots, 10$
belongs to
$\Ker(M_{3,4} -I_{3,4})$
if and only if 
\begin{gather*}
A^t(x_1 + x_2 + 3x_3) = 0, \\
x_6 = x_7 = x_8 = 0, \quad
x_4 = x_5 = x_3, \quad
x_9 = x_3 -x_2,\quad
x_{10} = x_2.
\end{gather*}
Hence  
the map $s_{\L^{S_G}}((n^l)_{l \in \Zp}) = \sum_{i=1}^{m(l)}$ is surjective.
\end{proof}
Since the cyclic six-term exact sequence \eqref{eq:Khom6}
is rephrased by \eqref{eq:6termL},
the upper right horizontal arrow
$\Extwz(\OSG) 
\longrightarrow \Z$
 in \eqref{eq:Khom6}
is surjective, so that we  have the exact sequences:
\begin{gather*}
0 \longrightarrow \Extsz(\OSG) 
\longrightarrow \Extwz(\OSG) 
\longrightarrow \Z \longrightarrow 0, \\
0 \longrightarrow \Extso(\OSG) 
\longrightarrow \Extwo(\OSG) 
\longrightarrow 0
\end{gather*}
which show that 
\begin{equation*}
\Extsz(\OSG) \oplus \Z  \cong \Extwz(\OSG), \qquad
\Extso(\OSG) \cong \Extwo(\OSG).
\end{equation*}
Since we know that 
$\Extwz(\OSG) \cong \Z^N \oplus (\Ker(A)\text{ in } \Z^N)$
by \cite{MaActaSci2014},
we have the strong extension groups $\Extsi(\OSG), i= 0,1$ in the following way.
\begin{proposition}\label{prop:MarkovcodeExts}
Let $G$ be an essential finite directed graph.
Suppose that its transition matrix $A$ of $G$ is aperiodic. 
Let $\OSG$ be the simple purely infinite $C^*$-algebra of the Markov coded system for $G$.
Then we have
\begin{equation*}
\Extsz(\OSG) \cong \Z^{N-1} \oplus (\Ker(A) \text{ in } \Z^N),
\qquad
\Extso(\OSG) \cong \Z^N/ A\Z^N.
\end{equation*}
\end{proposition}
\subsection{Dyck shifts}
We will compute the extension groups
for the $C^*$-algebra $\ODNmin$
associated to the minimal presentation $\LDNmin$
of the 
Dyck shift $D_N$ for $N \ge 2$.
Let 
$\Sigma^- =\{\alpha_1,\dots,\alpha_N \}$
and 
$\Sigma^+ =\{\beta_1,\dots,\beta_N \}$
be two kinds of finite sets.
Put
$\Sigma = \Sigma^+ \cup \Sigma^-$.
The Dyck shift 
$D_N$ is defined to be the subshift over 
$\Sigma$ in the following way.
Equip the set of finite words of $\Sigma$ with a monoid structure by
\begin{equation}\label{eq:Dyckmonoid}
\alpha_i \beta_j =
\begin{cases}
{\bf 1} & \text{ if } i=j, \\
0 & \text{ if } i\ne j.
\end{cases} 
\end{equation}
Let ${\frak F}_N$ be the set of finite words 
$(\gamma_1,\dots, \gamma_n)$ of $\Sigma$ 
such that
the poduct 
$\gamma_1\cdots\gamma_n$
 is zero in the monoid. 
The Dyck shift 
$D_N$ is defined by the subshift over $\Sigma$ whose fobiddern words are 
${\frak F}_N$. 
This means that 
$D_N$ is the set of bi-infinite sequences $(\gamma_n)_{n\in \Z}$
 of $\Sigma$ such that 
 $(\gamma_n, \dots, \gamma_{n+1}, \dots, \gamma_{n+k})$
does not belong to
${\frak F}_N$ for all $n \in \Z$ and $k \in \N$.
The subshift has a unique minimal presentation of $\lambda$-graph system
called the Cantor horizon $\lambda$-graph system  written
${\frak L}_{D_N}^{\operatorname{Ch}}$ (\cite{KMDocMath2003}).
In this paper, 
we call it the minimal presentation and write it as 
$\LDNmin =(V^\min, E^\min, \lambda^\min, \iota^\min)$. 
It is constructed as in the following way.
The vertex set $V_l^\min$ is 
\begin{equation*}
V_l^\min =\{ (\beta_{\nu_1},\dots, \beta_{\nu_l}) \in (\Sigma^+)^l 
\mid (\nu_1,\dots,\nu_l) \in \{1,\dots,N\}^l\}. 
\end{equation*}
A labeled edge labeled $\beta_j$
 is defined as a directed edge from the vertex 
$(\beta_j,\beta_{\nu_1},\dots, \beta_{\nu_{l-1}}) \in V_l^\min$
to the vertex  
$(\beta_{\nu_1},\dots, \beta_{\nu_l},\beta_{\nu_{l+1}}) \in V_{l+1}^\min$.
A labeled edge labeled $\alpha_j$
 is defined as a directed edge from the vertex 
$(\beta_{\nu_1},\dots, \beta_{\nu_{l}}) \in V_l^\min$
to the vertex 
$(\beta_{\nu_0}, \beta_{\nu_1},\dots, \beta_{\nu_l}) \in V_{l+1}^\min$
if and only if  $j=\nu_0$.
The set of such edges are denotd by $E^\min_{l,l+1}$.
The map $\iota: V_{l+1}^\min \longrightarrow  V_{l}^\min$
is defined by 
$\iota(\beta_{\nu_1},\dots, \beta_{\nu_l},\beta_{\nu_{l+1}})
=(\beta_{\nu_1},\dots, \beta_{\nu_l}).$
We then have a $\lambda$-graph system $\L_{D_N}^\min$.
It is irreducible and locally contracting in the sense of \cite{MaJMAA2021}
so that the $C^*$-algebra 
$\mathcal{O}_{\LDNmin}$ 
is a unital separable nuclear simple purely infinite
$C^*$-algebra.
It is written as $\ODNmin$.
The $\K$-groups were computed as 
\begin{equation*}
\K_0(\mathcal{O}_{{D_N}^\min}) = \Z/N\Z \oplus C(\calC,\Z), \,
\qquad \K_1(\mathcal{O}_{{D_N}^\min}) = 0
\end{equation*}
in \cite{KMDocMath2003},
where $C(\calC, \Z)$ denotes the abelian group of integer valued continuous functions on a Cantor set
$\calC$.
By the universal coefficient theorem
\begin{gather*}
0 \longrightarrow \Ext_\Z^1(\K_0(\A),\Z) 
\longrightarrow \Extwo(\A)
\longrightarrow \Hom_\Z(\K_1(\A),\Z)
\longrightarrow 0, \\
0 \longrightarrow \Ext_\Z^1(\K_1(\A),\Z) 
\longrightarrow \Extwz(\A)
\longrightarrow \Hom_\Z(\K_0(\A),\Z)
\longrightarrow 0
\end{gather*}
for a separable unital nuclear $C^*$-algebra $\A$ proved by L. Brown \cite{Brown84},
 we know the following proposition. 
\begin{proposition}
$\Extwo(\ODNmin) = \Z/N\Z,\qquad  
\Extwz(\ODNmin) = \Hom_\Z(C(\calC, \Z),\Z).
$
\end{proposition}
In this section, 
we will  compute the other extension groups
$\Extso(\ODNmin)$ and $\Extsz(\ODNmin)$.
We will consider the cases for $N=2$,
so that the alphabet of $D_2$ 
is $\Sigma =\{\alpha_1,\alpha_2, \beta_1,\beta_2\}$.
Let $(A_{l,l+1}, I_{l,l+1})_{l\in \Zp}$
be the structre matrix for the minimal presentation 
$\LDNmin$.
The cardinality $m(l)$ of the vertex set $V_l^\min$ is $2^l$.
As in \cite{KMDocMath2003},
define $m(l) \times m(l+1)$ matrices 
$J_{l,l+1}, \, K_{l,l+1},\,   L_{l,l+1},\,\, l\in \Zp$ such as
\begin{gather*}
J_{0,1} = [1,1], \qquad
J_{1,2} =
\begin{bmatrix}
1 & 0 & 1 & 0\\
0 & 1 & 0 & 1
\end{bmatrix},
\qquad
J_{2,3} =
\begin{bmatrix}
1 & 0 & 0 & 0 & 1 & 0 & 0 & 0 \\
0 & 1 & 0 & 0 & 0 & 1 & 0 & 0 \\
0 & 0 & 1 & 0 & 0 & 0 & 1 & 0 \\
0 & 0 & 0 & 1 & 0 & 0 & 0 & 1
\end{bmatrix}, \dots \\
J_{l,l+1} = [J_{l,l+1}(i,j)]_{i=1,2,\dots,m(l)}^{j=1,2,\dots,m(l+1)}
\quad 
\text{ where } 
J_{l,l+1}(i,j) = 
\begin{cases}
1 & \text{ if } j=i, \, \,  m(l)+ i,  \\
0 & \text{ otherwise, }  
\end{cases} 
\end{gather*}
\begin{gather*}
K_{0,1} = [1,1], \qquad
K_{1,2} =
\begin{bmatrix}
1 & 1 & 1 & 1\\
1 & 1 & 1 & 1
\end{bmatrix},\quad
K_{2,3} =
\begin{bmatrix}
1 & 1 & 1 & 1 & 0 & 0 & 0 & 0 \\
0 & 0 & 0 & 0 & 1 & 1 & 1 & 1 \\
1 & 1 & 1 & 1 & 0 & 0 & 0 & 0 \\
0 & 0 & 0 & 0 & 1 & 1 & 1 & 1
\end{bmatrix},
\dots \\
K_{l,l+1} = [K_{l,l+1}(i,j)]_{i=1,2,\dots,m(l)}^{j=1,2,\dots,m(l+1)}
\quad \\
\text{ where } 
K_{l,l+1}(i,j) 
= 
\begin{cases}
1 & \text{ if } j=4 i -3, \, \, 4i-2,  \, \, 4i -1, \, \, 4i  
\text{ for } 1\le i \le m(l-1), \\
1 & \text{ if } j=4 i -3 -m(l-1), \, \, 4i-2-m(l-1),  \, \, \\
  & \hspace{3mm} 4i -1-m(l-1), \, \, 4i-m(l-1)  \text{ for } m(l-1) +1\le i \le m(l) \\
0 & \text{ otherwise, }  
\end{cases}
\end{gather*}
\begin{gather*}
L_{0,1}= [1,1], \qquad
L_{1,2} =
\begin{bmatrix}
1 & 1 & 0 & 0\\
0 & 0 & 1 & 1
\end{bmatrix},
\qquad
L_{2,3} =
\begin{bmatrix}
1 & 1 & 0 & 0 & 0 & 0 & 0 & 0 \\
0 & 0 & 1 & 1 & 0 & 0 & 0 & 0 \\
0 & 0 & 0 & 0 & 1 & 1 & 0 & 0 \\
0 & 0 & 0 & 0 & 0 & 0 & 1 & 1
\end{bmatrix},
\dots \\
L_{l,l+1} = [L_{l,l+1}(i,j)]_{i=1,2,\dots,m(l)}^{j=1,2,\dots,m(l+1)}
\quad 
\text{ where } 
L_{l,l+1}(i,j) = 
\begin{cases}
1 & \text{ if } j=2i-1, \,   2i,  \\
0 & \text{ otherwise. }  
\end{cases}
\end{gather*}
We directly see the following lemma.
\begin{lemma}\label{lem:IA}
$I_{l,l+1} = L_{l,l+1}$ and
$A_{l,l+1} = J_{l,l+1} + K_{l,l+1}
$ for
$l \in \Zp$
so that we have
$$
I_{l,l+1} - A_{l,l+1} = L_{l,l+1} - J_{l,l+1} - K_{l,l+1}, \qquad l \in \Zp.
$$ 
\end{lemma}
For example 
\begin{gather*}
I_{0,1} - A_{0,1} = [-2, -2], \qquad
I_{1,2} - A_{1,2} =
\begin{bmatrix}
-1 & 0 & -2 & -1\\
-1 & -2 & 0 & -1
\end{bmatrix},\\
I_{2,3} - A_{2,3} =
\begin{bmatrix}
-1 & 0  & -1 & -1 & -1 & 0  & 0  & 0 \\
 0 & -1 & 1  &  1 & -1 & -2 & -1 & -1 \\
-1 & -1 & -2 & -1 & 1  &  1 &-1  & 0 \\
0  & 0  & 0  & -1 &-1  & -1 & 0  & -1
\end{bmatrix}, \dots.
\end{gather*}
The following lemmas are straightforward.
\begin{lemma}\label{lem:Dyck1}
Let $\xi: \Z_I \longrightarrow \Z$ be the homomorphism
defined by
$\xi(([n_i^l]_{i=1}^{m(l)})_{l\in \Zp}) = \sum_{i=1}^{m(l)}n_i^l \in \Z$.
\begin{enumerate}
\renewcommand{\theenumi}{\roman{enumi}}
\renewcommand{\labelenumi}{\textup{(\theenumi)}}
\item The value $\sum_{i=1}^{m(l)} n_i^l$ 
for each $l \in \Zp$ does not depend on $l \in \Zp.$
\item $\Ker(\xi) = \Z_{I,0}$, so that 
$\Z_I/\Z_{I,0}$ is isomorphic to $\Z$.
 \end{enumerate}
\end{lemma}
\begin{lemma}\label{lem:Dyck2}
\begin{enumerate}
\renewcommand{\theenumi}{\roman{enumi}}
\renewcommand{\labelenumi}{\textup{(\theenumi)}}
\item 
$I_{l,l+1}\Z^{m(l+1)}_0 \subset \Z^{m(l)}_0, \quad A_{l,l+1}\Z^{m(l+1)}_0 \subset \Z^{m(l)}_0$ so that
$$(I_{l,l+1}- A_{l,l+1}) \Z^{m(l+1)}_0 \subset \Z^{m(l)}_0.$$
\item The diagram
\begin{equation*}
\begin{CD}
\Z^{m(l+1)} @>{\xi_{l+1}}>> \Z \\
@V{I_{l,l+1}}VV     \parallel \\ 
\Z^{m(l)} @>{\xi_l}>> \Z
\end{CD}
\end{equation*}
commutes.
\end{enumerate}
\end{lemma}
\begin{lemma}
For any $[n_i^l]_{i=1}^{m(l)} \in \Z_0^{m(l)}$
and $[m_i^l]_{i=1}^{m(l)} \in \Z_0^{m(l)}$,
there exists $[m_j^{l+1}]_{j=1}^{m(l+1)} \in \Z_0^{m(l+1)}$
such that 
\begin{enumerate}
\renewcommand{\theenumi}{\arabic{enumi}}
\renewcommand{\labelenumi}{\textup{(\theenumi)}}
\item $[m_i^l]_{i=1}^{m(l)} = I_{l,l+1} [m_j^{l+1}]_{j=1}^{m(l+1)},$
\item $[n_i^l]_{i=1}^{m(l)} = (I_{l,l+1}- A_{l,l+1}) [m_j^{l+1}]_{j=1}^{m(l+1)}.$
 \end{enumerate}
\end{lemma}
\begin{proof}
The condition (1) is rephrazed as  
\begin{equation*}
m_i^l = m_{2i-1}^{l+1} + m_{2i}^{l+1}, \qquad i =1,\dots, m(l),
\end{equation*}
where $m(l) = 2^l, \, m(l+1) = 2^{l+1}.$
Since
\begin{align*}
  & K_{l,l+1} [m_j^{l+1}]_{j=1}^{m(l+1)} \\
= & [m_1^{l+1}+ m_2^{l+1}+ m_3^{l+1}+ m_4^{l+1}, \dots, m_{m(l+1)-3}^{l+1}+ m_{m(l+1)-2}^{l+1}+ m_{m(l+1)-1}^{l+1}+ m_{m(l+1)}^{l+1}] \\
= & [m_1^{l}+ m_2^{l},  \dots, m_{m(l)-1}^{l}+ m_{m(l)}^{l}] 
\end{align*}
and
\begin{align*}
  & L_{l,l+1} [m_j^{l+1}]_{j=1}^{m(l+1)} \\ 
= & [m_1^{l+1}+ m_2^{l+1}, \dots, m_{m(l+1)-1}^{l+1}+ m_{m(l+1)}^{l+1}] 
=  [m_1^{l},  \dots, m_{m(l)}^{l}] 
\end{align*}
it suffices to show that 
for given $[n_i^l]_{i=1}^{m(l)}, [m_i^l]_{i=1}^{m(l)} \in \Z^{m(l)}_0,$
one may find 
$[m_j^{l+1}]_{j=1}^{m(l+1)} \in \Z^{m(l+1)}$
such that 
\begin{equation*}
m_i^l = m_{2i-1}^{l+1} + m_{2i}^{l+1}, \qquad
n_i^l = -m_i^{l+1} - m_{i + m(l)}^{l+1}, \qquad
i=1,\dots, m(l).
\end{equation*}
Since 
$I_{l,l+1} - A_{l,l+1} = L_{l,l+1}  -J_{l,l+1} - K_{l,l+1},$
this is possible 
because of the form of $J_{l,l+1}$
\end{proof}
Therefore we have
\begin{lemma}\label{lem:IAZI}
The equality 
$(I_{l,l+1}- A_{l,l+1}) \Z^{m(l+1)}_0 = \Z^{m(l)}_0$
holds  for each $l \in \Zp,$
so that we have
$(I - A_\L) \Z_{I,0} = \Z_{I,0}.$
\end{lemma}
We reach the following theorem.
\begin{theorem}\label{thm:mainextsdtwo}
$\Extso(\ODtmin) \cong \Z.$
\end{theorem}
\begin{proof}
We have $\Extso(\ODtmin) = \Z_I/ ( I - A_\L)\Z_{I,0}.$
By Lemma \ref{lem:IAZI} together with Lemma \ref{lem:Dyck1} (ii),
we obtain $\Extso(\ODtmin) \cong \Z.$
\end{proof}
Although the weak extension group 
$\Extwo(\ODtmin)$ for $\ODtmin$
had been computed to be $\Z/2\Z$
in \cite{KMDocMath2003} through $\K$-group computation $\K_*(\ODtmin)$ and the universal coefficient theorem,
we may give another proof  without using the K-group formulas in the following way.  
\begin{proposition}\label{prop:extsZ}
The diagram
\begin{equation*}
\begin{CD}
\Z  @>{\iota}>> \Extso(\ODtmin)  \\
\parallel @. @VV{\Indso}V      \\ 
\Z  @>{\hat{\iota}_\L}>> \Z_I/ ( I - A_\L)\Z_{I,0} =\Z  
\end{CD}
\end{equation*}
commutes, where $\hat{\iota}_\L(m)=-2 m$ for $m \in \Z$.
\end{proposition}
\begin{proof}
Recall that for $m = \sum_{j=1}^{m(l+1)} n_j^{l+1}$ with 
$n^{l+1} = [n_j^{l+1}]_{j=1}^{m(l+1)}$ and $(n^l)_{l \in \Zp} \in \Z_I$,
we have
\begin{align*}
 \hat{\iota}_\L(m)
 = & \xi_l((I - A_\L)[ [n_j^{l+1}]_{j=1}^{m(l+1)}]) \\
 = & \sum_{i=1}^{m(l)} 
     \sum_{j=1}^{m(l+1)} 
     (I_{l,l+1}(i,j) - A_{l,l+1}(i,j) )([ [n_j^{l+1}]_{j=1}^{m(l+1)}]).
 \end{align*}
Since 
$\sum_{i=1}^{m(l)}  (I_{l,l+1}(i,j) - A_{l,l+1}(i,j)) = -2$
for each $j=1,\dots,m(l)$, we have
\begin{equation*}
 \hat{\iota}_\L(m)
 =  \sum_{j=1}^{m(l+1)} (-2) n_j^{l+1} = -2 \sum_{j=1}^{m(l+1)} n_j^{l+1} = -2 m.
\hspace{3cm} \qed
\end{equation*}
\renewcommand{\qed}{}
\end{proof}
\begin{corollary}
The diagram
\begin{equation*}
\begin{CD}
0   @>>>   \Z   @>>>           \Extso(\ODtmin) @>>>   \Extwo(\ODtmin) @>>> 0   \\
@.       \parallel @.                    @V{\Indso}VV           @V{\Indwo}VV         @.  \\
0   @>>>   \Z   @>{\times(-2)}>> \Z             @>>>   \Z/2\Z         @>>> 0
\end{CD}
\end{equation*}
is commutative such that the vertical arrows $\Indso$ and $\Indwo$
are both isomorphic
so that we have
$\Extwo(\ODtmin) \cong \Z/2 \Z.$
\end{corollary}
As in the proof of 
\cite[Section 5]{KMDocMath2003}, 
one may easily generalize the above discussion 
of $D_2$ to general Dyck shifts $D_N, 2\le N \in \N$
to get the following theorem.
Since the proof of the generalization 
is direct and tedious, so we omit the proof.
\begin{theorem}\label{thm:mainDN}
There is a commutative diagram
\begin{equation*}
\begin{CD}
0   @>>>   \Z   @>>>             \Extso(\ODNmin) @>>>   \Extwo(\ODNmin) @>>> 0   \\
@.       \parallel @.                    @V{\Indso}VV           @V{\Indwo}VV         @.  \\
0   @>>>   \Z   @>{\times(-N)}>> \Z             @>>>   \Z/N\Z         @>>> 0
\end{CD}
\end{equation*}
such that the vertical arrows $\Indso$ and $\Indwo$
are both isomorphic
so that we have
$\Extso(\ODNmin) \cong \Z$ and $\Extwo(\ODNmin) \cong \Z/N \Z.$
\end{theorem}
\begin{corollary}\label{cor:DyckNExts}
$\Extsz(\ODNmin) = \Extwz(\ODNmin) = \Hom_\Z(C(\calC, \Z),\Z),$
where $ C(\calC, \Z)$ 
is the abelian group of integer valued continuous functions 
on a Cantor set $\calC$.
\end{corollary}
\begin{proof}
There exists a  cyclic six-term exact sequence
\begin{equation}\label{eq:6termextON}
\begin{CD}
\Extsz(\ODNmin) @>{}>> \Extwz(\ODNmin) @>{\partial}>> \Z \\
@AAA  @.   @VV{\times N}V  \\
0 @<{}<< \Extwo(\ODNmin)  @<{}<< \Extso(\ODNmin)   
\end{CD}
\end{equation}
for the $C^*$-algebra $\ODNmin.$
Since the map 
$\Z \overset{\times N}{\longrightarrow}\Extso(\ODNmin) $ is injective,
the connecting map
$\partial: \Extwo(\ODNmin) \longrightarrow \Z$ is the zero map,
so that 
we have 
$\Extwo(\ODNmin) \cong \Extso(\ODNmin).$
As in \cite[Section 5]{KMDocMath2003},
$\Extso(\ODNmin) \cong \Hom_\Z(C(\calC, \Z),\Z),$
we get the assertioon.  
\end{proof}
\medskip

{\it Acknowledgment:}
The author is indebted to Taro Sogabe  for stimulating conversations and discussions
on extension groups of $C^*$-algebras..
This work was supported by JSPS KAKENHI 
Grant Numbers 19K03537, 24K06775.

\end{document}